\documentclass[11pt]{article}

\usepackage[backend=bibtex]{biblatex}
\usepackage{latexsym}
\usepackage{amssymb}
\usepackage{amsthm}
\usepackage{amscd}
\usepackage{amsmath}
\usepackage{mathrsfs}
\usepackage{graphicx}
\usepackage{hyperref}
\usepackage{shuffle}
\usepackage{ytableau}
\usepackage{stmaryrd}

\usepackage[all]{xy}
\input xy \xyoption{frame}
\xyoption{dvips}

\usepackage[colorinlistoftodos]{todonotes}
\usetikzlibrary{chains,scopes,decorations.markings}

\usepackage{mathdots}

\theoremstyle{definition}
\newtheorem* {theorem*}{Theorem}
\newtheorem* {conjecture*}{Conjecture}
\newtheorem{theorem}{Theorem}[section]

\theoremstyle{definition}

\newtheorem* {example*}{Example}

\newtheorem{lemma}[theorem]{Lemma}
\theoremstyle{definition}
\newtheorem{definition}[theorem]{Definition}
\theoremstyle{definition}
\newtheorem* {notation}{Notation}
\newtheorem{conjecture}[theorem]{Conjecture}
\newtheorem{proposition}[theorem]{Proposition}
\newtheorem{corollary}[theorem]{Corollary}

\newtheorem{remark}[theorem]{Remark}
\theoremstyle{definition}
\newtheorem {example}[theorem]{Example}
\theoremstyle{definition}

\theoremstyle{definition}

\theoremstyle{definition}

\theoremstyle{definition}

\newcommand{\ytabc}[2]{
\ytableausetup{boxsize = #1,aligntableaux=center}
{\small\begin{ytableau}  #2  \end{ytableau}}
}

\newcommand{\ytabb}[1]{
\ytableausetup{boxsize = .55cm,aligntableaux=center}
{\small\begin{ytableau}  #1  \end{ytableau}}
}

\newcommand{\ytab}[1]{
\ytableausetup{boxsize = .4cm,aligntableaux=center}
{\small\begin{ytableau}  #1  \end{ytableau}}
}

\def\modu{\ (\mathrm{mod}\ }

\def\({\left(}
\def\){\right)}

\newcommand{\cC}{\mathcal{C}}
\newcommand{\cD}{\mathcal{D}}

\def\cX{\mathcal{X}}

\def\NN{\mathbb{N}}

\def\ZZ{\mathbb{Z}}

\def\ch{\mathrm{ch}}

\def\barr{\begin{array}}
\def\earr{\end{array}}
\def\ba{\begin{aligned}}
\def\ea{\end{aligned}}
\def\be{\begin{equation}}
\def\ee{\end{equation}}

\def\quand{\quad\text{and}\quad}

\def\quord{\quad\text{or}\quad}

\newcommand{\gl}{\mathfrak{gl}}
\newcommand{\Sym}{\operatorname{Sym}}

\def\PP{\mathbb{P}}

\def\ben{\begin{enumerate}}
\def\een{\end{enumerate}}

\def\bei{\begin{itemize}}
\def\eei{\end{itemize}}

\def\cE{\mathcal E}

\def\e{\textbf{e}}

\newcommand{\xRightarrow}[2][]{\ext@arrow 0359\Rightarrowfill@{#1}{#2}}

\newcommand{\cB}{\mathcal{B}}

\def\arcstart{\ \xy<0cm,-.06cm>\xymatrix@R=.1cm@C=.10cm }
\newcommand{\arcstartc}[1]{\ \xy<0cm,-.15cm>\xymatrix@R=.1cm@C=#1cm}

\def\bar{\overline}

\definecolor{darkred}{rgb}{0.7,0,0} 
\newcommand{\defn}[1]{{\color{darkred}\emph{#1}}} 

\newcommand{\weight}{\operatorname{wt}}
\def\BB{\mathbb{B}}
\def\q{\mathfrak{q}}

\def\Tab{\textsf{Tab}}
\def\DTab{\textsf{DecTab}}

\def\ShTab{\textsf{ShTab}}
\def\SShTab{\textsf{ShTab}^+}

\def\row{\mathsf{row}}
\def\revrow{\mathsf{revrow}}

\def\D{\mathsf{D}}
\def\SD{\mathsf{SD}}

\def\SetShTab{\mathsf{SetShTab}}
\def\SetDTab{\mathsf{SetDecTab}}

\def\GP{G\hspace{-0.5mm}P}
\def\GQ{G\hspace{-0.5mm}Q}

\def\Thighest{T^{\mathsf{highest}}}

\def\wGP{\Sigma}

\def\sqgln{\sqrt{\gl_n}}
\def\sqqn{\sqrt{\q_n}}

\def\SetTab{\mathsf{SetTab}}
\def\SS{\mathbb{S}}
\def\SVWords{\mathsf{SetWords}}
\def\svcol{\mathsf{col}_{{SV}}}
\def\svrow{\row_{{SV}}}
\def\svrevrow{\revrow_{{SV}}}
\def\Words{\mathsf{Words}}
\def\Sym{\mathsf{Sym}}

\numberwithin{equation}{section}
\allowdisplaybreaks[1]
\UseCrayolaColors

\usepackage{fullpage}

\begin{document}
\title{Crystals for set-valued decomposition tableaux}
\author{
Eric MARBERG\thanks{
Department of Mathematics, HKUST, {\tt emarberg@ust.hk}.
}
\and
Kam Hung TONG\thanks{
Department of Mathematics, HKUST, {\tt khtongad@connect.ust.hk}.
}
}

\date{}

\maketitle

\begin{abstract}
We describe two crystal structures on set-valued decomposition tableaux. These provide the first examples of interesting ``$K$-theoretic'' crystals on shifted tableaux. Our first crystal is modeled on a similar construction of Monical, Pechenik, and Scrimshaw for semistandard (unshifted) set-valued tableaux. Our second crystal is adapted from the ``square root''  operators introduced by Yu on the same set. Neither of our shifted crystals is normal, but we conjecture that our second construction is connected with a unique highest weight element. These results lead to partial progress on a conjectural formula of Cho--Ikeda for $K$-theoretic Schur $P$-functions.   We also study a new category of ``square root crystals'' that includes our second construction and Yu's set-valued tableau crystals as examples. We observe that Buch's formula for the coefficients expanding products of symmetric Grothendieck functions has a simple description in terms of the tensor product for this category. 
\end{abstract}

\tableofcontents

\section{Introduction}

Work by Grantcharov et al. \cite{GJKKK15,GJKKK,MR2628823}
defines a family of 
\defn{crystals} for the queer Lie superalgebra $\q_n$. 
Concretely, these objects consist of certain directed acyclic graphs with labeled weights and weighted vertices. There is a natural tensor product for $\q_n$-crystals,
as well as a \defn{standard $\q_n$-crystal}.
The connected crystals that appear in tensor powers of standard crystal,
as well as arbitrary disjoint unions of such crystals, are called \defn{normal}.

It is shown in \cite{GJKKK} that every connected normal $\q_n$-crystal
has a unique highest weight element, whose weight is a strict partition $\lambda$
with at most $n$ parts, and that any two such crystals with the same highest weight are isomorphic. Conversely, the authors in \cite{GJKKK} provide a simple, explicit construction of a connected normal $\q_n$-crystal with  any feasible highest weight $\lambda$:
this is given by introducing natural crystal operators on all \defn{decomposition tableaux} of shape $\lambda$.

A \defn{decomposition tableau} is a certain filling of the shifted shape of a strict partition $\lambda$ by positive integers.
In such tableaux, 
each row must be a \defn{hook word}, meaning a weakly decreasing sequence followed by a (possibly empty) strictly increasing sequence.
Additionally, consecutive rows in a decomposition tableau must avoid the patterns represented in French notation as
 \[
 \ytabb{
 \none & \none & \cdots & b  \\
 \none & a  & \cdots & \  
 },
\quad
 \ytabb{
 \none & \cdots  & c & \cdots & b  \\
\ & \cdots & a & \cdots & \ 
 },
  \quad
 \ytabb{
 \none & \cdots  &x     \\
y & \cdots & z  
 },
 \quand
 \ytabb{
 \none & \cdots  &  & \cdots & x  \\
\ & \cdots & y & \cdots & z 
 }
 \]
  for all $a\leq b \leq c$ and $x<y<z$.
See
Section~\ref{prelim-sect} for the precise definitions.

Our purpose here is to describe new crystal structures on ``set-valued'' generalizations of decomposition tableaux.
Several authors  (for example, \cite{HawkesScrimshaw,MPS,PT2022,Yu}) have recently studied crystal structures on unshifted set-valued tableaux. The characters of these crystals give $K$-theoretic symmetric functions of independent interest. It has been an open problem to extend such constructions to shifted tableaux. 

Addressing this problem, we first show that the family of \defn{set-valued decomposition tableaux} 
$\SetDTab_n(\lambda)$
has an abstract $\q_n$-crystal structure. The elements of $\SetDTab_n(\lambda)$ are the set-valued shifted tableaux (that is, fillings of a shifted shape by nonempty subsets of $\{1,2,\dots,n\}$)  whose distributions all belong to $\DTab_n(\lambda)$. 
By a \defn{distribution} of a set-valued tableau, we mean any tableau of the same shape formed by replacing each set-valued entry by  one of its elements.
For example, we can draw an element of $\SetDTab_4(\lambda)$ for $\lambda =(3,2)$ in French notation as
\[
\ytableausetup{boxsize = .8cm,aligntableaux=center}
\begin{ytableau}
\none & 1\,2 & 3 
\\
4 & 3 & 2\,3\,4  
\end{ytableau} 
\quad
\text{since}\quad
\left\{\ytab{\none & 1 & 3 \\ 4 & 3 & 2}, 
\ytab{\none & 2 & 3 \\ 4 & 3 & 3},
\ytab{\none & 1 & 3 \\ 4 & 3 & 4},
\ytab{\none & 2 & 3 \\ 4 & 3 & 2},
\ytab{\none & 1 & 3 \\ 4 & 3 & 3},
\ytab{\none & 2 & 3 \\ 4 & 3 & 4}\right\} \subset \DTab_4(\lambda).
\]

Our $\q_n$-crystal structure on $\SetDTab_n(\lambda)$ is formally similar to the one in \cite{MPS} for unshifted set-valued tableaux. The details of its construction appear in Section~\ref{a-sm-sect}.

Cho and Ikeda \cite{IkedaMisc} have conjectured that the weight generating function for set-valued decomposition tableaux recovers the \defn{$K$-theoretic Schur $P$-function} $\GP_\lambda$ introduced in \cite{IkedaNaruse} (see Conjecture~\ref{Ik-conj}). By interpreting this generating function
as the character of our $\q_n$-crystal $\SetDTab_n(\lambda)$,
we are able to
 prove that it is at least
equal to $\GP_\lambda$ plus a (possibly infinite) linear combination of $\GP_\mu$'s with $|\mu|>|\lambda|$ (see Corollary~\ref{GP-cor}).

Our first crystal structure on $\SetDTab_n(\lambda)$, like the one in \cite{MPS}, is typically disconnected.
Yu \cite{Yu} introduced a modified set of ``square root'' crystal operators on set-valued words that square to a seminormal crystal structure. These operators determine a connected crystal when restricted to semistandard set-valued tableaux of a fixed shape (identified with their column reading words). It turns out that
Yu's crystal operators also make sense as operators on set-valued decomposition tableaux (now using the reverse row reading word), and provide an alternate ``square root'' crystal structure on these objects (see Theorem~\ref{sqqn-thm2}).
We conjecture that the resulting crystal, when equipped with additional ``square root'' queer crystals operators, is actually connected (see Conjecture~\ref{sqqn-conj}).

We also discuss an axiomatic framework to contain these constructions.
Specifically, we introduce definitions of abstract \defn{$\sqgln$- and $\sqqn$-crystals},
which include Yu's crystal of semistandard set-valued tableaux and our second crystal of set-valued decomposition tableaux as special cases. These ``crystals'' possess a natural tensor product and a plausible choice of a ``standard object,''
which give rise to families of \defn{normal $\sqgln$- and $\sqqn$-crystals}.
The crystals do not generally have 
highest weight properties that are as nice as in the classical case,
but they nevertheless display several interesting features (see, for example, Theorem~\ref{sq-tensor-thm2}, Theorem~\ref{buch-cor}, and Proposition~\ref{sqrt-char-prop}).
We discuss these nuances in Section~\ref{sqrt-crystal-tensor-sect}
and present several related conjectures in Section~\ref{conj-sect}.

\subsection*{Acknowledgments}

This work was partially supported by Hong Kong RGC grants 16306120 and 16304122. We thank Zach Hamaker, Takeshi Ikeda, Joel Lewis, Brendan Pawlowski, and Travis Scrimshaw for many useful discussions. We are especially grateful to Tianyi Yu for several comments on and corrections to the first draft of this paper, 
which led to the proofs of Theorems~\ref{sq-tensor-thm2}
 and \ref{buch-cor}.

\section{Preliminaries}\label{prelim-sect}

This section contains some background material on shifted tableaux and 
crystals from \cite{BumpSchilling,GJKKK,Macdonald}.
Let $\NN =\{0,1,2,\dots\}$ and $\PP=\{1,2,3,\dots\}$.
Fix $ n \in \NN$ and let $[n] = \{1,2,\dots,n\}$.

\subsection{Tableaux}\label{sh-sect}

A \defn{partition} is a weakly decreasing sequence of integers $\lambda = (\lambda_1 \geq \lambda_2 \geq \dots > 0)$ with finite sum. The nonzero numbers $\lambda_i$ are the \defn{parts} of $\lambda$; if these are all distinct then $\lambda$ is \defn{strict}.  Let $\ell(\lambda)$ be the number of nonzero parts of $\lambda$.
The \defn{diagram} of $\lambda$ is the set \[\D_\lambda := \{ (i,j) :  i\in  [\ell(\lambda)] \text{ and } j \in [\lambda_i]\}.\]
When $\lambda = (\lambda_1 > \lambda_2 > \dots > 0)$ is a strict partition, its
  \defn{shifted diagram}  is 
\[ \SD_\lambda:= \{ (i,i+j-1) : i\in  [\ell(\lambda)] \text{ and } j \in [\lambda_i]\} = \{ (i,i+j-1) : (i,j) \in \D_\lambda\}.\]
For the moment, we define a \defn{tableau} of shape $\lambda$  to be any map $ \D_\lambda \to \{1,2,3,\dots\}$. We often view this object as a filling of the positions in $\D_\lambda$, drawn as boxes, by positive integers.
Likewise, a \defn{shifted tableau} of shape $\lambda$  is a map $ \SD_\lambda \to \{1'<1<2'<2<\dots\}$. Here, one can view the \defn{primed numbers} $1'<2'<3'<\dots$ either as formal symbols
or as the half-integers $i' := i - \frac{1}{2}$.

If $T$ is a (shifted) tableau, then we  write $(i,j) \in T$ to indicate that $(i,j)$ belongs to the domain of $T$
and we let $T_{ij}$ denote the value assigned to this position.
We draw tableaux in French notation, so that row indices increase bottom to top and column indices increase left to right.
For example,
\be\label{tableau-ex}
R = \ytab{   3 & 3& 7 \\ 1 & 2 & 2 & 6},
\quad S = \ytab{  \none& 3 & 5& 7 \\ 1 & 2 & 4 & 6},
\quand 
T = \ytab{ \none & 2' & 2 & 4' \\ 1' & 1 & 1 & 4'}  
\ee
are all tableaux of shape $\lambda=(4,3)$, with $R_{23}=7$, $S_{23} = 5$, and $T_{23} = 2$. 
The tableau $R$ is unshifted while $S$ and $T$ are shifted. The \defn{(main) diagonal} of a shifted tableau
is the set of boxes $(i,j)$ in its domain with $i=j$.

A (shifted or unshifted) tableau is \defn{semistandard} if its  rows and columns are weakly increasing,
such that no primed entry is repeated in any row and no unprimed entry is repeated in any column.
The examples 
in \eqref{tableau-ex}
are all semistandard.
We write  $\Tab(\lambda)$ for the set of all semistandard tableaux of shape $\lambda$,
  $\SShTab(\lambda)$
 for the set of all semistandard shifted tableaux of shape $\lambda$,
 and $ \ShTab(\lambda)$ for the subset of elements in $ \SShTab(\lambda)$
with no primed entries on the diagonal.

\begin{notation}
Throughout, if $\mathsf{FAMILY}(\lambda)$ is any family of tableaux associated to a partition $\lambda$,
then we write $\mathsf{FAMILY}_n(\lambda)$ for the subset of $T \in \mathsf{FAMILY}(\lambda)$ with all entries at most $n$.
   \end{notation}

If $T$ is a (shifted) tableau, then   set
$ x^T := x_1^{a_1} x_2^{a_2} \cdots x_n^{a_n}$ where $a_k$ is the number of times $k$ or $k'$ appears in $T$.
The \defn{Schur function} of an arbitrary partition $\lambda$ is
\be\label{schur-eq}
\textstyle
s_\lambda := \sum_{T \in \Tab(\lambda)} x^T.
\ee
Similarly, 
the \defn{Schur $P$- and $Q$-functions}  of a strict partition $\lambda$ are  
\be\label{PQ-eq}
\textstyle
P_\lambda := \sum_{T \in \ShTab(\lambda)} x^T
\quand
Q_\lambda := \sum_{T \in \SShTab(\lambda)} x^T = 2^{\ell(\lambda)} P_\lambda .\ee
These power series all belong to the ring of bounded degree symmetric functions $\Sym$.
The Schur functions $\{ s_\lambda\}$ are a $\ZZ$-basis for $\Sym$,
while
 $\{P_\lambda\}$ and $\{Q_\lambda\}$ (as $\lambda$ varies over all strict partitions) are $\ZZ$-bases for two distinct subrings of $\Sym$  \cite[Chapter III, \S8]{Macdonald}.

The Schur $P$-functions have a second tableau generating function formula.
A \defn{hook word} is a finite sequence of positive integers $w=w_1w_2\cdots w_n$ 
such that $w_1 \geq w_2 \geq \dots \geq w_m < w_{m+1} < w_{m+2} <\dots <w_n$ for some $m\in[n]$.
 Given such a hook word, let ${w\downarrow} := w_1w_2\cdots w_m$ denote the \defn{decreasing part} and 
 let ${w\uparrow} := w_{m+1}w_{m+2}\cdots w_n$ denote the \defn{increasing part}.

 A \defn{decomposition tableau} of strict partition shape $\lambda$ 
 is a map $T : \SD_\lambda \to \{1,2,3,\dots\}$ such that if $\rho_i$ denotes row $i$ of $T$,
 then (1) each $\rho_i$ is a hook word and (2)  $\rho_i$ is a hook subword of maximal length in $\rho_{i+1}\rho_i$ for each $i \in [\ell(\lambda)-1]$.
This definition follows \cite{GJKKK} but differs from \cite{CNO,Serrano}, which uses the opposite weak/strict inequality convention for hook words.
Let $\DTab(\lambda)$ be the set of all decomposition tableaux of shape $\lambda$.
Then we have
\be\textstyle P_\lambda = \sum_{T \in \DTab(\lambda)} x^T \ee
by  \cite[Thm.~2.17]{Serrano} and \cite[Thm.~3.9]{CNO}.

 \begin{example}
We have
$ 
\ytab{ \none & 1 & 1 \\ 2 & 2 & 1 } \in  \DTab_3((3,2))
$ but
$  \ytab{ \none & 1 & 1 \\ 2 & 2 & 3 } \notin  \DTab_3((3,2)).
$
The second example is not a decomposition tableau because its row reading word 
$\rho_2\rho_1  = 11223$ contains the hook subword $1123$ which is longer than $\rho_1=223$.
 \end{example}
 
 It is useful to reformulate
  the maximal hook subword condition as follows:

 \begin{lemma}[{\cite[Prop.~2.3]{GJKKK}}] \label{decomptab-characterisation}
 Let $T$ be shifted tableau of shape $\lambda$ whose rows are each hook words.
 Then $T$ is a decomposition tableaux if and only if none of the following conditions holds for any $i \in [\ell(\lambda)-1]$ and $j,k \in [\lambda_{i+1}]$:
 \ben
 
  \item[(a)] $T_{i,i} \leq T_{i+1,i+k}$ or $ T_{i,i+j} \leq T_{i+1,i+k}  \leq  T_{i+1,i+j}$ when $j<k$, 
 \item[(b)] $T_{i+1,i+k} < T_{i,i} < T_{i,i+k}$ or $T_{i+1,i+k} < T_{i,i+j} < T_{i,i+k}$ when $j<k$.
  \een
 That is, we forbid rows $i$ and $i+1$ of $T$ from having configurations of entries 
 \[
 \ytabb{
 \none & \none & \cdots & b  \\
 \none & a  & \cdots & \  
 },
\quad
 \ytabb{
 \none & \cdots  & c & \cdots & b  \\
\ & \cdots & a & \cdots & \ 
 },
  \quad
 \ytabb{
 \none & \cdots  &x     \\
y & \cdots & z  
 },
 \quord
 \ytabb{
 \none & \cdots  &  & \cdots & x  \\
\ & \cdots & y & \cdots & z 
 }
 \]
  with $a\leq b \leq c$ and $x<y<z$.
  Here, the leftmost boxes are on the main diagonal and the ellipses ``$\cdots$'' indicate sequences of zero or more columns.
 \end{lemma}

\subsection{Abstract crystals}

Let $\cB$ be a set with maps $\weight  :  \cB\to \ZZ^n$
and 
$e_i,f_i  :  \cB \to \cB \sqcup \{0\}$  for $i \in [n-1]$,
where $0 \notin \cB$.
Write $\e_1,\e_2,\dots,\e_n\in\ZZ^n$ for the standard basis.

\begin{definition}
\label{crystal-def}
The set $\cB$ is a
 \defn{$\gl_n$-crystal}
if 
for all $i \in [n-1]$ and $b,c \in \cB$
 it holds that  $e_i(b) = c $ if and only if $ f_i(c) = b$, 
in which case $ \weight(c) = \weight(b) + \e_{i} -\e_{i+1}$.
\end{definition}

Assume $\cB$ is a $\gl_n$-crystal. Then the maps $e_i$ and $f_i$ encode a directed graph with vertex set $\cB$, to be called the \defn{crystal graph},
with an edge $b \xrightarrow{i} c$ if and only if $ f_i(b)=c$.
Define the \defn{string lengths} $\varepsilon_i, \varphi_i  :  \cB \to \{0,1,2,\dots\}\sqcup \{\infty\}$ 
by \be\label{string-eqs}
\varepsilon_i(b) := \sup\left\{ k\geq 0 \mid e_i^k(b) \neq 0\right\}
\text{ and }
\varphi_i(b) := \sup\left\{ k \geq 0: f_i^k(b) \neq 0\right\}.
\ee

\begin{definition} The $\gl_n$-crystal $\cB$ is \defn{seminormal}
if the string lengths always take finite values 
such that $\varphi_i(b) - \varepsilon_i(b) = \weight(b)_i - \weight(b)_{i+1}$ for all $b \in \cB$.
\end{definition}

If $\cB$ is finite then its \defn{character} is   
$ \ch(\cB) := \sum_{b \in \cB} x^{\weight(b)}$ where $x^{\weight(b)} := \prod_{i\in[n]} x_i^{\weight(b)_i}.$
The character is symmetric in $x_1,x_2,\dots,x_n$ if $\cB$ is seminormal \cite[\S2.6]{BumpSchilling}.

We refer to $\weight$ as the \defn{weight map}, to each $e_i$ as a \defn{raising operator}, and to each $f_i$ as a \defn{lowering operator}.
Each connected component of the crystal graph of $\cB$ may be viewed as a $\gl_n$-crystal
by restricting the weight map and crystal operators; these objects are called \defn{full subcrystals}.

\begin{example}\label{st-ex1}
The \defn{standard $\gl_n$-crystal} $\BB_n = \left\{ \boxed{i}: i \in [n]\right\}$ has
crystal graph
\[
    \begin{tikzpicture}[xscale=1.8, yscale=1,>=latex,baseline=(z.base)]
    \node at (0,0.0) (z) {};
      \node at (0,0) (T0) {$\boxed{1}$};
      \node at (1,0) (T1) {$\boxed{2}$};
      \node at (2,0) (T2) {$\boxed{3}$};
      \node at (3,0) (T3) {${\cdots}$};
      \node at (4,0) (T4) {$\boxed{n}$};
      \draw[->,thick]  (T0) -- (T1) node[midway,above,scale=0.75] {$1$};
      \draw[->,thick]  (T1) -- (T2) node[midway,above,scale=0.75] {$2$};
      \draw[->,thick]  (T2) -- (T3) node[midway,above,scale=0.75] {$3$};
      \draw[->,thick]  (T3) -- (T4) node[midway,above,scale=0.75] {$n-1$};
     \end{tikzpicture}
\quad\text{with }\weight(\boxed{i}):=\e_i.
\]
Its character is $\ch(\BB_n) = x_1+x_2+x_3+\dots+x_n$.
  \end{example}
  
  If $\cB$ and $\cC$ are $\gl$-crystals
then the set   $\cB \otimes \cC := \{ b\otimes c : b\in \cB,\ c\in \cC\}$
 of formal tensors has a unique $\gl_n$-crystal structure
 (which is seminormal if $\cB$ and $\cC$ are seminormal)
in which
$
\weight(b\otimes c) := \weight(b) + \weight(c)
$
and  
\be\label{ebc-eq}
e_i(b\otimes c) := \begin{cases}
b \otimes e_i(c) &\text{if }\varepsilon_i(b) \leq \varphi_i(c) \\
e_i(b) \otimes c &\text{if }\varepsilon_i(b) > \varphi_i(c)
\end{cases}
\quand
f_i(b\otimes c) := \begin{cases}
b \otimes f_i(c) &\text{if }\varepsilon_i(b) < \varphi_i(c) \\
f_i(b) \otimes c &\text{if }\varepsilon_i(b) \geq \varphi_i(c)
\end{cases}
\ee
for $i \in [n-1]$,
where we set $b\otimes 0 = 0\otimes c = 0$  \cite[\S2.3]{BumpSchilling}.
This follows the ``anti-Kashiwara convention,''
which reverses the tensor product order  in \cite{GJKKK15,GJKKK}.
The natural map $\cB \otimes (\cC \otimes \cD) \to (\cB \otimes \cC) \otimes \cD$ is a crystal isomorphism,
 so we can dispense with   parentheses in iterated tensor products.

A $\gl_n$-crystal is \defn{normal} if each of its full subcrystals is isomorphic to a full subcrystal of $\BB_n^{\otimes m}$ for some $m \in \NN$. Such crystals have particularly nice properties, such as Schur positive characters; see \cite[Thms.~3.2 and 8.6]{BumpSchilling} and \cite{Stembridge2003}.

A \defn{word} is a finite sequence of positive integers.
Let $\Words_n(m)$ be the set of $m$-letter words $w=w_1w_2\cdots w_m$ with $w_i \in [n]$. 
We identify each word $w\in \Words_n(m)$ with the tensor $w_1\otimes w_2 \otimes \cdots \otimes w_m $ in order to view $\Words_n(m)$ as a crystal isomorphic to $ \BB_n^{\otimes m}$.
This allows us to evaluate $\weight(w)$, $e_i(w)$, and $f_i(w)$ for $i \in[n-1]$ using the definition of the $\gl_n$-crystal $\BB_n^{\otimes m}$. The weight of $w$ under this convention is the usual integer vector whose $i$th component is the number of letters equal to $i$.

\begin{remark}\label{signa-rem}
The following \defn{signature rule} \cite[\S2.4]{BumpSchilling})
can be used to compute $e_i(w)$ and $f_i(w)$.
Suppose 
$w=w_1w_2\cdots w_m \in \Words_m(n)$ is a word
and $i \in [n-1]$. Mark each entry $w_k = i$ by a right parenthesis ``)"  and each entry $w_j = i+1$ 
by a left parenthesis ``(". The \defn{$i$-unpaired indices} in $w$ are the indices $j \in [m]$ with 
$w_j \in \{i, i+1\}$ that are not the positions of matching parentheses. 
\begin{itemize}
\item If no $i$-unpaired index $j$ of $w$ has $w_j=i+1$ then $e_i(w) = 0$.
\newline
Otherwise, if $j$ is the smallest such index, then $e_i(w) = w_1 \cdots (w_j-1)\cdots w_m$.

\item If no $i$-unpaired index $j$ of $w$ has $w_j=i$ then $f_i(w) = 0$.
\newline
Otherwise, if $j$ is the largest such index, then $f_i(w) = w_1 \cdots (w_j+1)\cdots w_m$.

\end{itemize}
\end{remark}

\subsection{Queer crystals}

Grantcharov et al. developed a theory of crystals for
the \defn{queer Lie superalgebra} $\q_n$ in \cite{GJKKK15,GJKKK,MR2628823}, which we review here. Assume $n\geq 2$ and 
let $\cB$ be a $\gl_n$-crystal with additional maps $e_{\bar 1},f_{\bar 1} : \cB \to \cB \sqcup \{0\}$.
Define $\varepsilon_{ \bar 1}, \varphi_{ \bar 1} : \cB \to \NN\sqcup \{\infty\}$ 
as in  \eqref{string-eqs}   with $i=\bar 1$.
Below, we say that one map $\phi : \cB \to \cB\sqcup \{0\}$ \defn{preserves} another map $\eta : \cB \to \cX$
if $\eta(\phi(b)) = \eta(b)$ whenever $\phi(b) \neq 0$.

\begin{definition}
\label{q-crystal-def}
The $\gl_n$-crystal $\cB$ is a
 \defn{$\q_n$-crystal} 
  if both of the following hold:
\ben
\item[(a)]  $e_{\bar 1}$, $f_{\bar 1}$ commute with $e_i$, $f_i$ while preserving $\varepsilon_i$, $\varphi_i$ for all $3\leq i \leq n-1$,
\item[(b)] if $b,c \in \cB$ then
$e_{\bar 1}(b) = c$ if and only if $f_{\bar 1}(c) = b,$
in which case
$\weight(c) = \weight(b) + \e_{1} -\e_{2}.$
 \een
\end{definition}

Assume $\cB$ is a $\q_n$-crystal. The corresponding \defn{$\q_n$-crystal graph}  has vertex set $\cB$ and edges $b \xrightarrow{i} c$ whenever $f_i(b) =c$ for any $i \in \{\bar 1,1,2,\dots,n-1\}$.

\begin{definition} 
A $\q_n$-crystal $\cB$ is \defn{seminormal}
if it is seminormal as a $\gl_n$-crystal and for all $b \in \cB$ one has
$\weight(b)\in \NN^n
$
and
$
\varphi_{\bar 1}(b) + \varepsilon_{\bar 1}(b)\leq 1$, with strict inequality
if and only if $\weight(b)_1 =\weight(b)_{2}= 0$.
\end{definition}

If $\cB$ is a finite seminormal $\q_n$-crystal then $\ch(B)$ is a $\ZZ$-linear combination of Schur $P$-polynomials $P_\lambda(x_1,x_2,\dots,x_n)$
by \cite[Prop. 2.5]{Marberg2019b}.

\begin{example}\label{st-ex2}
The \defn{standard $\q_n$-crystal} $\BB_n = \left\{ \boxed{i}: i \in [n]\right\}$ has
crystal graph
\[
    \begin{tikzpicture}[xscale=1.8, yscale=1,>=latex,baseline=(z.base)]
    \node at (0,0.0) (z) {};
      \node at (0,0) (T0) {$\boxed{1}$};
      \node at (1,0) (T1) {$\boxed{2}$};
      \node at (2,0) (T2) {$\boxed{3}$};
      \node at (3,0) (T3) {${\cdots}$};
      \node at (4,0) (T4) {$\boxed{n}$};
      \draw[->,darkred,thick]  (T0.15) -- (T1.165) node[midway,above,scale=0.75] {$ \overline1$};
      \draw[->,thick]  (T0) -- (T1) node[midway,below,scale=0.75] {$ 1$};
      \draw[->,thick]  (T1) -- (T2) node[midway,below,scale=0.75] {$ 2$};
      \draw[->,thick]  (T2) -- (T3) node[midway,below,scale=0.75] {$ 3$};
      \draw[->,thick]  (T3) -- (T4) node[midway,below,scale=0.75] {$ {n-1}$};
     \end{tikzpicture}
  \quad\text{with }\weight(\boxed{i}):=\e_i.
  \]
\end{example}

Suppose $\cB$ and $\cC$ are $\q_n$-crystals.
The set 
$\cB \otimes \cC$ already has a $\gl_n$-crystal structure. 
There is a unique way of viewing this object as a $\q_n$-crystal
with
\be\label{q-tensor1}
  e_{\overline 1}(b\otimes c) := \begin{cases} 
 b \otimes e_{\overline 1}(c)&\text{if }e_{\overline1}(b) = f_{\overline1}(b) = 0
 \\
  e_{\overline 1}(b) \otimes c
&\text{otherwise}
 \end{cases}
\ee
and
\be\label{q-tensor2}
  f_{\overline 1}(b\otimes c) := \begin{cases} 
 b \otimes f_{\overline 1}(c)&\text{if }e_{\overline1}(b) = f_{\overline1}(b) = 0
 \\
  f_{\overline 1}(b) \otimes c
&\text{otherwise}
 \end{cases}
\ee
where it is again understood that $b\otimes 0 = 0\otimes c = 0$ \cite[Thm.~1.8]{GJKKK}.
The natural map $\cB \otimes (\cC \otimes \cD) \to (\cB \otimes \cC) \otimes \cD$ is again an isomorphism,
and if $\cB$ and $\cC$ are seminormal then so is $\cB\otimes \cC$.

A $\q_n$-crystal is \defn{normal} if each of its full subcrystals is isomorphic to a full subcrystal of $\BB_n^{\otimes m}$ for some $m \in \NN$. As in the $\gl_n$-case, these crystals have many nice properties, such as Schur $P$-positive characters; see \cite[Thm.~2.5 and Cor.~4.6]{GJKKK}.

\begin{remark} The formulas below for 
the action $e_{\bar 1}$ and $f_{\bar 1}$  
act on $w=w_1w_2\cdots w_m\in\Words_n(m) \cong \BB_n^{\otimes m}$
are easy to check from \eqref{q-tensor1} and \eqref{q-tensor2} by induction on $m$; see also \cite{GHPS}:
\begin{itemize}
\item If $w$ has no $2$'s or if a $1$ appears before the first $2$, then $e_{\bar 1}(w)=0$.
\newline
Otherwise, $e_{\bar 1}(w)$ is formed from $w$ by changing the first $2$  to $1$.

\item If $w$ has no $1$'s or if a $2$ appears before the first $1$, then $f_{\bar 1}(w)=0$.
\newline
Otherwise, $f_{\bar 1}(w)$ is formed from $w$ by changing the first $1$ to $2$.

\end{itemize}
The   operators $e_i$ and $f_i$ for $i \in [n-1]$ act on $w$ exactly as in the $\gl_n$-case. 
\end{remark}

 The \defn{row reading word} of a shifted tableau $T$ is 
the word $\row(T)$ formed by reading the rows from left to right, but starting with last row.
The \defn{reverse row reading word} of  $T$ is  the reversal
of $\row(T)$; we denote this by $\revrow(T)$. 
\begin{example}
We have $\row\(\ytab{ \none & 2 & 1 \\ 2 & 2 & 3 }\) =21223$
and
$\revrow\(\ytab{ \none & 2 & 1 \\ 2 & 2 & 3 }\) =32212.$
\end{example}

A \defn{crystal embedding} is a weight-preserving injective map $\phi : \cB \to \cC$ between crystals 
that commutes with all crystal operators, in the sense that $\phi(e_i(b)) = e_i(\phi(b))$ and $\phi(f_i(b)) = f_i(\phi(b))$
for all $b \in \cB$ when we set $\phi(0) =0$.

\begin{example}\label{revrow-thm}
Let $\lambda$ be a strict partition of $m$ with at most $n$ parts.
There is a unique (normal) $\q_n$-crystal structure on $\DTab_n(\lambda)$ 
that makes $\revrow : \DTab_n(\lambda)\to \Words_n(m) \cong \BB_n^{\otimes m}$
into a $\q_n$-crystal embedding \cite[Thm.~2.5(a)]{GJKKK}.
The weight map for this crystal has $x^{\weight(T)} = x^T$.
\end{example}

 \section{First construction}

This section contains our first main result, which describes a ``set-valued'' analogue 
of the crystal on decomposition tableaux just introduced in Example~\ref{revrow-thm}.
We review some preliminaries on \defn{set-valued tableaux} in Section~\ref{sv-sect}
before defining the relevant crystal operators.

\subsection{Set-valued tableaux}\label{sv-sect}

A \defn{set-valued tableaux} of partition shape $\lambda$ is a filling $T$ of the diagram $\D_\lambda$ by finite, nonempty sets of $\PP=\{1,2,3,\dots\}$.
A \defn{set-valued shifted tableaux} of strict partition shape $\lambda$ is likewise a filling $T$ of   $\SD_\lambda$ by finite, nonempty sets of  $ \{1'<1<2'<2<\dots\}$.
A \defn{distribution} of a (shifted) set-valued tableau a tableau of the same shape formed by replacing every
set-valued entry by one of its elements. 
For example,
\[
\ytableausetup{boxsize = .65cm,aligntableaux=center}
T =
\begin{ytableau}
\none & 1 & 3 
\\
4 & 1'3 & 234  
\end{ytableau}
\quad\text{has 6 distributions}\quad
\ytableausetup{boxsize = .4cm,aligntableaux=center}
\begin{ytableau}
\none & 1 & 3 
\\
4 & 1' & 2  
\end{ytableau},
\
\begin{ytableau}
\none & 1 & 3 
\\
4 & 1' & 3  
\end{ytableau},
\
\begin{ytableau}
\none & 1 & 3 
\\
4 & 1' & 4  
\end{ytableau},
\
\begin{ytableau}
\none & 1 & 3 
\\
4 & 3 & 2  
\end{ytableau},
\
\begin{ytableau}
\none & 1 & 3 
\\
4 & 3 & 3  
\end{ytableau},
\
\begin{ytableau}
\none & 1 & 3 
\\
4 & 3 & 4  
\end{ytableau}.
\]
A set-valued (shifted) tableau is \defn{semistandard} if its distributions are all semistandard according to the definition in Section~\ref{sh-sect}.

Let $\SetTab(\lambda)$ be the set of all semistandard set-valued (unshifted) tableaux of shape $\lambda$.
Then let $\SetShTab^+(\lambda)$ be the set of all semistandard set-valued shifted tableaux of shape $\lambda$, and define $\SetShTab(\lambda)$ to be
the subset of such tableaux with no primed numbers appearing in any set-valued entries on the main diagonal.
Define $\SetTab_n(\lambda)$, $\SetShTab^+_n(\lambda)$, and $\SetShTab_n(\lambda)$
using our convention in Section~\ref{sh-sect}.
The last three sets are nonempty if and only if $\ell(\lambda)\leq n$.

Given any set-valued tableau $T$, let 
\[
\textstyle 
\weight(T) := (a_1,a_2,a_3,\dots)\quand x^T := \prod_{k \in \PP} x_k^{a_k} = x^{\weight(T)}
\]
 where $a_k$ is how many times $k$ or $k'$ appears in $T$.
The \defn{symmetric Grothendieck function} of $\lambda$  is
 \be G_\lambda =  \sum_{T \in \SetTab(\lambda)} x^T.\ee
Similarly, Ikeda and Naruse's
  \defn{$K$-theoretic Schur $P$- and $Q$-functions} of $\lambda$ are
\be
\GP_\lambda = \sum_{T \in \SetShTab(\lambda)} x^T 
\quand
\GQ_\lambda = \sum_{T \in \SetShTab^+(\lambda)} x^T. 
\ee
Often in the literature the definitions of these power series involve a bookkeeping parameter $\beta$. Here,
for simplicity, we have set $\beta=1$. The more general definition
is recovered by making the variable substitutions $x_i \mapsto \beta x_i$ and then dividing by $\beta^{|\lambda|}$.

\begin{remark}
Work of Buch \cite{Buch2002} shows that each $G_\lambda$ is symmetric with lowest degree term given by the Schur function $s_\lambda$.
Similarly, $\GP_\lambda $ and $\GQ_\lambda$ are both Schur positive symmetric functions, though of unbounded degree \cite[Thms.~3.27 and 3.40]{MarScr}.

Specializations of  $\GP_\lambda $ and $\GQ_\lambda$  give equivariant $K$-theory representatives for Schubert varieties in the maximal isotropic Grassmannians of orthogonal and symplectic types \cite[Cor.~8.1]{IkedaNaruse}.
These symmetric functions have another geometric interpretation
as the stable limits of $K$-theory representatives for certain orbit closures in the complete flag variety \cite{MP2019a,MP2019b}
as well as
remarkable positivity properties \cite{CTY,HKPWZZ,LM2022,PechenikYong} (see \cite[\S4.6]{Mar2022} for a survey of results and open problems).
They also have other formulas besides the tableau generating functions given above; see \cite{IkedaNaruse,Iwao,NakagawaNaruse}.
\end{remark}

Define a \defn{set-valued decomposition tableau} of strict partition shape $\lambda$
to be a set-valued shifted tableau whose distributions are each decomposition tableaux of shape $\lambda$.
Let $\SetDTab(\lambda)$ be the set of all such tableaux and let  $\SetDTab_n(\lambda)$ 
be the subset with all entries  at most $n$. A large amount of computation supports the following conjecture:

\begin{conjecture}[Cho--Ikeda \cite{IkedaMisc}] \label{Ik-conj}
It holds that $\GP_\lambda = \sum_{T \in \SetDTab(\lambda)} x^T$.
\end{conjecture}

Currently we do not know of any analogous conjectural formula
for $\GQ_\lambda$ involving a set-valued analogue of decomposition tableaux.
It would be interesting to find such a formula.

%

%

\subsection{A seminormal crystal}\label{a-sm-sect}

Crystals for $\gl_n$ have been identified on  
$\SetTab_n(\lambda)$ in \cite{MPS} and \cite{Yu} (see also \cite{HawkesScrimshaw,PT2022}). Both of these prior constructions  have analogues for set-valued decomposition tableaux, which can be used to derive a weaker form of Conjecture~\ref{Ik-conj}. The simpler of these crystals is described in this section.

Fix a strict partition $\lambda$ and  $T \in \SetDTab(\lambda)$.
 The \defn{reverse row reading word} of $T$ is the word $\revrow(T)$ formed by 
 iterating over the boxes of $T$ in the reverse row reading word order
 (starting with the last box of the first row and proceeding row by row, reading each row right to left), 
and listing the entries of each box in decreasing order.
Define $\weight(T) = \weight(\revrow(T))$.
For example, if 
\[
\ytableausetup{boxsize = .65cm,aligntableaux=center}
T =
\begin{ytableau}
\none & 1 & 3 
\\
4 & 123 & 234  
\end{ytableau}
\]
then
$\revrow(T) = 432  321  4  3  1$
and
$\weight(T) = (2,2,3,2).$

Fix $i \in \PP$.
Consider the word formed by replacing each $i$ in $\revrow(T)$ by a ``)"  and each $i+1$ 
by a ``(". We define a letter in $\revrow(T)$ to be \defn{$i$-unpaired}
if it is equal to $i$ or $i+1$ but does not belong to 
a matching pair of parentheses in this modified word.

\begin{definition}
Given $i \in \PP$ and $T \in \SetDTab(\lambda)$,
form $e_i(T) \in \SetDTab(\lambda)\sqcup \{0\}$  as follows:
\begin{itemize}
\item Define $e_i(T) = 0$ if there are no $i$-unpaired letters in $\revrow(T)$ equal to $i+1$.
\item Otherwise, suppose the first $i$-unpaired $i+1$ in $\revrow(T)$ is in box $(x,y)$.

\ben 
\item[(a)] Form $e_i(T) $ from $T$ by  changing the $i+1$ in box $(x,y)$ to $i$ if this  yields a set-valued decomposition tableau.
For example,
\[
e_2 : \begin{ytableau}
\none & 1 & 2
\\
3 & {\textcolor{red}{13}}& 123  
\end{ytableau} 
\mapsto
\begin{ytableau}
\none & 1 & 2
\\
3 & 12 & 123  
\end{ytableau}.
\]

\item[(b)] Otherwise, some box $(a,b)$ preceding $(x,y)$ in the reverse row reading word order
has  $\{i,i+1\}\subseteq T_{ab}$. If $(a,b)$ is the last such box, then 
form $e_i(T)$  by removing $i+1$ from $T_{ab}$ and adding $i$ to $T_{xy}$.
The box $(a,b)$ either has $a=x$ and $b>y$, as in the example
\[
e_3 : \begin{ytableau}
\none & 1 & 2 
\\
{\textcolor{red}4} & 1 & {\textcolor{blue}{34}  }
\end{ytableau} 
\mapsto
\begin{ytableau}
\none & 1 & 2 
\\
34  & 1  &3   
\end{ytableau},
\]
or has $a=x-1$ and $b<y$, as in the example
\[
e_2 : \begin{ytableau}
\none & 1 & {\textcolor{red}3} 
\\
4 & {\textcolor{blue}{123}} & 234  
\end{ytableau} 
\mapsto
\begin{ytableau}
\none & 1 & 23 
\\
4 & 12 & 234  
\end{ytableau}.
\]
\een
\end{itemize}
\end{definition}

\begin{definition}
Given $i \in \PP$ and $T \in \SetDTab(\lambda)$,
form  $f_i(T) \in \SetDTab(\lambda)\sqcup \{0\}$  as follows:
\begin{itemize}
\item Define $f_i(T) = 0$ if there are no $i$-unpaired letters  in $\revrow(T)$ equal to $i$.
\item Otherwise, suppose the last $i$-unpaired $i$ in $\revrow(T)$ is in box $(x,y)$ of $T$.

\ben 
\item[(a)] Form $f_i(T) $ from $T$ by  changing the $i$ in box $(x,y)$ to $i+1$ if this  yields a set-valued decomposition tableau.
For example,
\[
f_2 : 
\begin{ytableau}
\none & 1 & 2
\\
3 & {\textcolor{red}{12}} & 123  
\end{ytableau}
\mapsto
\begin{ytableau}
\none & 1 & 2
\\
3 & 13& 123  
\end{ytableau} .\]

\item[(b)] Otherwise, some box $(a,b)$ following $(x,y)$ in the reverse row reading word order
has  $\{i,i+1\}\subseteq T_{ab}$. If $(a,b)$ is the first such box, then 
form $f_i(T)$  by removing $i$ from $T_{ab}$ and adding $i+1$ to $T_{xy}$.
The box $(a,b)$ either has $a=x$ and $b<y$, as in the example
\[
f_3 :
\begin{ytableau}
\none & 1 & 2 
\\
{\textcolor{blue}{34}}  & 1  &{\textcolor{red}3}   
\end{ytableau}
\mapsto
 \begin{ytableau}
\none & 1 & 2 
\\
4& 1 & 34
\end{ytableau},
\]
or has $a=x+1$ and $b>y$, as in the example
\[
f_2 :
\begin{ytableau}
\none & 1 & {\textcolor{blue}{23}} 
\\
4 & {\textcolor{red}{12}} & 234  
\end{ytableau} 
\mapsto
\begin{ytableau}
\none & 1 &3
\\
4 & 123 & 234  
\end{ytableau} .
\]
\een
\end{itemize}
\end{definition}

A \defn{multiset-valued decomposition tableau}
 is defined in the same way as a set-valued decomposition tableau, only
 we allow entries to be multisets. 
 Let $\SetDTab^{\ast}(\lambda)$ be the set of multiset-valued decomposition tableaux
of shape $\lambda$ with the property that only the number $1$ appears more than once 
in any given box.

\begin{definition}
Form $e^\ast_{\bar 1}(T)$ and $f^\ast_{\bar 1}(T)$ 
from $T \in \SetDTab^\ast(\lambda)$ as follows:
\ben
\item[(a)] If $2 \notin \revrow(T)$ or the first $2\in\revrow(T)$ is after a $1$, then $e^\ast_{\bar 1}(T)=0$.
\newline
Otherwise, change 
the first $2$ in $\revrow(T)$  to $1$. For example,
\[
e^\ast_{\bar 1} : 
\ytabc{.7cm}{
\none & 11 & 2
\\
3 & 13 & {\textcolor{red}{123}}
}
\mapsto
\ytabc{.7cm}{
\none & 11 & 2
\\
3 & 13 & 113  
}
\quand
e^\ast_{\bar 1}: \ytabc{.7cm}{
\none & 11 & 2
\\
3 & 13 & 113  
} \mapsto 0.
\]

\item[(b)] If $1\notin \revrow(T)$ or the first $1\in\revrow(T)$ is after a $2$, then $f^\ast_{\bar 1}(T)=0$.
\newline
Otherwise, change 
the first $1$ in $\revrow(T)$ to $2$.
For example, \[
f^\ast_{\bar 1} : 
\ytabc{.7cm}{
\none & 111 & 2
\\
3 & 13 &  {\textcolor{red}{13}}
}
\mapsto
\ytabc{.7cm}{
\none & 111 & 2
\\
3 & 13 &23
}
\quand
f^\ast_{\bar 1}: 
\ytabc{.7cm}{
\none & 1 & 2
\\
3 & 13 &123
}
 \mapsto 0.
\]
\een
\end{definition}

\begin{proposition}
The operations $e^\ast_{\bar 1}$ and $f^\ast_{\bar 1}$ define maps 
$\SetDTab^\ast(\lambda) \to \SetDTab^\ast(\lambda) \sqcup \{0\},$
and 
if $T,U \in \SetDTab^\ast(\lambda)$ then $e^\ast_{\bar 1}(T) = U$ if and only if $T = f^\ast_{\bar 1}(U)$.
\end{proposition}

\begin{proof}
The claim that $e^\ast_{\bar 1}(T) = U$ if and only if $T = f^\ast_{\bar 1}(U)$ is clear from the definitions.
It remains 
to check the less evident property that $e^\ast_{\bar 1}(T) $ and $f^\ast_{\bar 1}(T) $
both belong to $\SetDTab^\ast(\lambda) \sqcup \{0\}$.

For this, 
choose $T \in \SetDTab^\ast(\lambda)$ with $e^\ast_{\bar 1}(T) \neq 0$.
Consider the box of $T$ containing the first $2$ in the reverse row reading word order.
On every distribution of $T$
containing the $2$  in this box, the crystal operator $e_{\bar 1}$ from Example~\ref{revrow-thm}
acts by changing this $2$ to $1$. It follows that $e^\ast_{\bar 1}(T) \in \SetDTab^\ast(\lambda)$.
Likewise, if $f^\ast_{\bar 1}(T) \neq 0$ and we consider the box of $T$ containing the first $1$ in the reverse row reading word order,
then $f_{\bar 1}$ acts on every distribution of $T$ containing the $1$ in this box by changing this $1$ to $2$, so 
 $f^\ast_{\bar 1}(T) \in \SetDTab^\ast(\lambda)$.
\end{proof}

\begin{definition}
Define $f_{\bar 1}:\SetDTab(\lambda) \to \SetDTab(\lambda)\sqcup\{0\}$
to be the restriction of $f_{\bar 1}$. Then
let $e_{\bar 1}$ be the map $\SetDTab(\lambda) \to \SetDTab(\lambda)\sqcup\{0\}$
with $e_{\bar 1}(T) = e_{\bar 1}^\ast(T)$ if this is not a multiset-valued tableau, and with $e_{\bar 1}(T) =0$ otherwise.
\end{definition}

Equivalently, $e_{\bar 1}$
acts as zero on set-valued decomposition tableaux
for which the first box in the reverse reading word order containing $1$ or $2$ contains both $1$ and $2$
(notice that $f_{\bar 1}$ also has this property); on all other $T\in\SetDTab(\lambda)$ we have 
$e_{\bar 1}(T)=e_{\bar1}^\ast(T)$.

\begin{theorem}\label{sv-thm}
Let $\lambda$ be a strict partition with at most $n$ parts, so that $\SetDTab_n(\lambda)$ is nonempty.
Then for the operators $e_{\bar 1}, e_1,e_2,\dots,e_{n-1}$ and $f_{\bar 1},f_1,f_2,\dots,f_{n-1}$ given above, 
 $\SetDTab_n(\lambda)$ is a $\q_n$-crystal
that is seminormal as a $\gl_n$-crystal.
\end{theorem}

We defer the proof to Section~\ref{sv-proof-sect} and discuss some applications here.

\begin{remark}
The crystal $\SetDTab_n(\lambda)$ is not typically $\q_n$-seminormal as it can occur that
  $\varphi_{\bar 1}(T) + \varepsilon_{\bar 1}(T)=0$
outside the case when $\weight(T)_1 = \weight(T)_2 =0$.
But we would not expect this set to have any seminormal $\q_n$-crystal structure, 
as this would imply that its character $\sum_{T \in \SetDTab(\lambda)} x^T$ is a formal linear combination of Schur $P$-functions,
which is not true of $\GP_\lambda$.

More surprisingly, $\SetDTab_n(\lambda)$ is not always a normal $\gl_n$-crystal; see Figure~\ref{non-stembridge-fig}.
Thus we cannot deduce that 
 $\sum_{T \in \SetDTab(\lambda)}  x^T$ is Schur positive,
  although this is expected in view of Conjecture~\ref{Ik-conj} as  $\GP_\lambda$ is Schur positive \cite[Thm.~3.27]{MarScr}. 
Even when the $\q_n$-crystal $\SetDTab_n(\lambda)$ is normal, it is often quite disconnected.
See Figures~\ref{setdtab-fig1} and \ref{setdtab-fig2} for examples.
\end{remark}

\begin{figure}
\centerline{\input{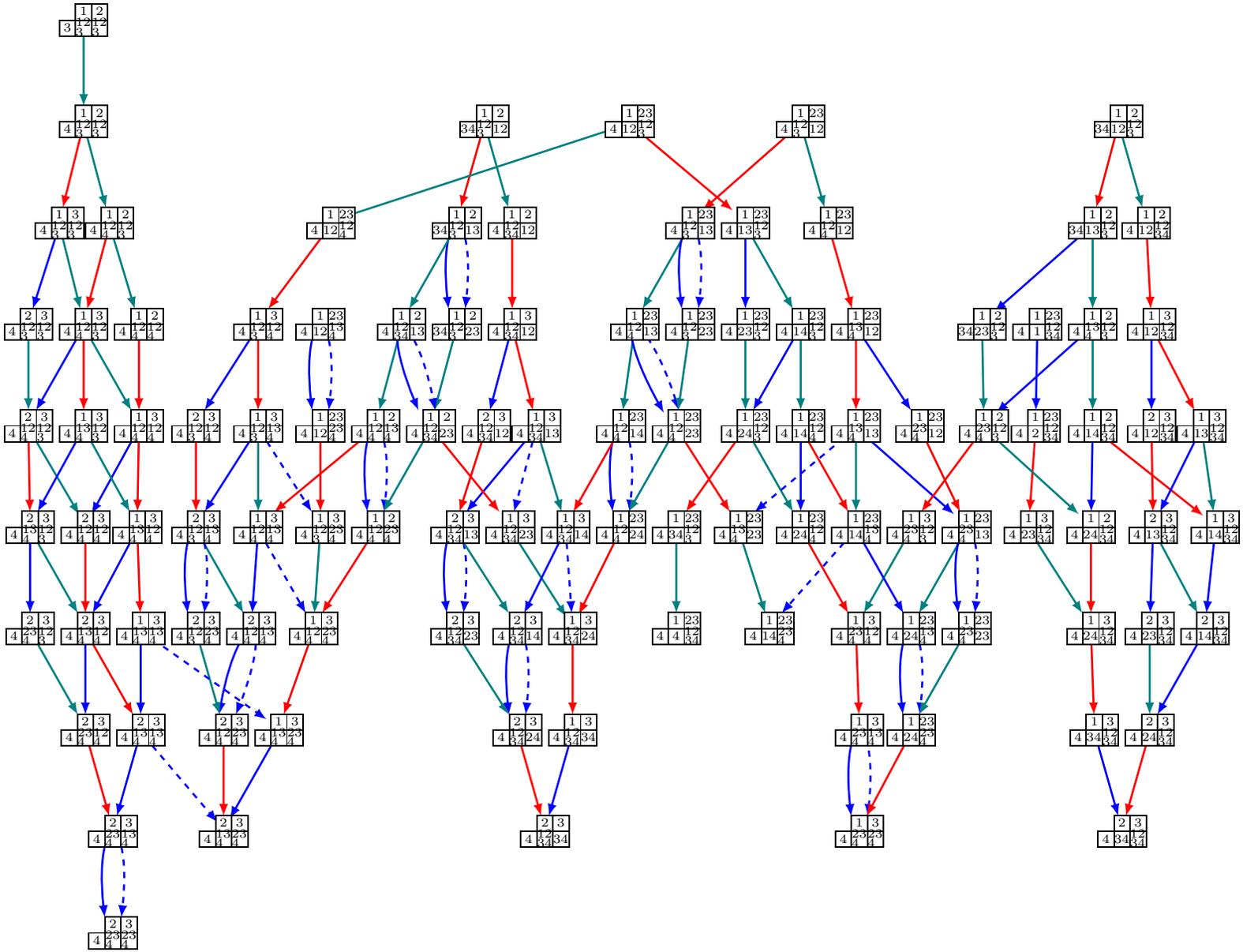}}
\caption{A connected component of the $\q_4$-crystal of set-valued decomposition tableaux 
$\SetDTab_4(\lambda)$ for $\lambda=(3,2)$ that is not normal as a $\gl_n$-crystal.
Solid blue, red, and green arrows respectively indicate 
$\xrightarrow{\ 1\ } $, $ \xrightarrow{\ 2\ } $, and $ \xrightarrow{\ 3\ } $ edges, while dotted blue arrows
indicate $\xrightarrow{\ \bar 1\ }$ edges.
}
\label{non-stembridge-fig}
\end{figure}

\begin{figure}
\centerline{\input{q-sv-dec2.tex}}
\caption{Crystal graph of the $\q_3$-crystal $\SetDTab_3(\lambda)$ for $\lambda=(2)$.
Here,
solid blue and red arrows respectively indicate 
$ \xrightarrow{\ 1\ } $ and $ \xrightarrow{\ 2\ } $ edges while dashed blue arrows 
indicate $ \xrightarrow{\ \bar 1\ }$ edges.
}
\label{setdtab-fig1}
\end{figure}

\begin{figure}
\centerline{\input{q-sv-dec21.tex}}
\caption{Crystal graph of the $\q_3$-crystal $\SetDTab_3(\lambda)$ for $\lambda=(2,1)$.
Here,
solid blue and red arrows respectively indicate 
$ \xrightarrow{\ 1\ } $ and $ \xrightarrow{\ 2\ } $ edges while dashed blue arrows 
indicate $ \xrightarrow{\ \bar 1\ }$ edges.
}
\label{setdtab-fig2}
\end{figure}

We can at least use Theorem~\ref{sv-thm} to show that $\sum_{T \in \SetDTab(\lambda)} x^T$ is symmetric, and a little more.
A polynomial or formal power series $f$ in an ordered sequence of variables $\{x_i\}$ 
satisfies the
 \defn{$K$-theoretic $Q$-cancelation property} 
 if for all variable indices $i< j$ the substitution
\[f(x_1,\dots,x_{i-1}, t, x_{i+1},\dots,x_{j-1}, \tfrac{-t}{1+t}, x_{j+1},\dots)\]
 does not depend on $t$.
When $f$ is symmetric, to check this property it suffices to take $i=1$ and $j=2$.

The symmetric functions in $ \ZZ\llbracket x_1,x_2,\dots\rrbracket$
with the $K$-theoretic $Q$-cancelation property are the ones 
that may be (uniquely) expressed as a possibly infinite $\ZZ$-linear combinations
of
 $\GP$-functions \cite[Prop.~3.4]{IkedaNaruse}.
In finitely many variables, 
the polynomials $\GP_\lambda(x_1,x_2,\dots,x_n)$, with $\lambda$ ranging over all strict partitions having at most $n$ parts, form a $\ZZ$-basis for the subring of symmetric elements in $\ZZ[x_1,x_2,\dots,x_n]$
with the $K$-theoretic $Q$-cancelation property \cite[Thm.~3.1]{IkedaNaruse}.
 
\begin{proposition}\label{kca-prop}
The power series $\sum_{T \in \SetDTab(\lambda)} x^T$ is a symmetric function with the $K$-theoretic $Q$-cancelation property
and lowest degree term $P_\lambda$.

\end{proposition}

\begin{proof}
Let  $\wGP_\lambda  :=  \sum_{T \in \SetDTab(\lambda)} x^T$.
This power series is symmetric since 
each polynomial $\wGP_\lambda(x_1,x_2,\dots,x_n)$ is the character of a seminormal $\gl_n$-crystal, namely
 $\SetDTab_n(\lambda)$. The lowest degree term of $\wGP_\lambda$ is 
 evidently $P_\lambda=  \sum_{T \in \DTab(\lambda)} x^T$.

To check the $K$-theoretic $Q$-cancelation property,
observe that
\[
\wGP_\lambda(\tfrac{x_1}{1-x_1}, -x_2,x_3,x_4,\dots) =  \sum_{T \in \SetDTab^\ast(\lambda)} (-1)^{c_2(T)} x^T
\]
where $c_i(T)$ denotes the number of times that $i$ appears in $T$.
We have $e^\ast_{\bar 1}(T) = f^\ast_{\bar 1}(T) = 0$ if and only if $c_1(T) = c_2(T) = 0$. 
If $e^\ast_{\bar 1}(T) \neq 0$ then 
$c_2(T) = c_2(e^\ast_{\bar 1}(T)) + 1$ and $ x^T =  x^{e^\ast_{\bar 1}(T)}  x_2 / x_1,$
in which case 
$
\((-1)^{c_2(T)} x^T + (-1)^{c_2(e^\ast_{\bar 1}(T))} x^{e^\ast_{\bar 1}(T)}\)\big|_{x_1=x_2} = 0.
$
Therefore 
\[
\wGP_\lambda(\tfrac{u}{1-u}, -u,x_3,x_4,\dots) =  \sum_{\substack{T \in \SetDTab^\ast(\lambda)\\ c_1(T) = c_2(T) = 0}}   x^T
= \wGP_\lambda( x_3,x_4,\dots).\]
Setting $u = \frac{t}{1+t}$ turns this into the $K$-theoretic $Q$-cancelation property.
\end{proof}

Since every $\GP_\lambda$ has lowest degree term $P_\lambda$, we
get this weaker form of Conjecture~\ref{Ik-conj}:

\begin{corollary}\label{GP-cor}
If  $\lambda $  is a strict partition 
then $\sum_{T \in \SetDTab(\lambda)} x^T$ is equal to $\GP_\lambda$ plus a possibly infinite $\ZZ$-linear combination
of power series $\GP_\mu$ with $|\mu|  > |\lambda|$.
\end{corollary}

\subsection{Proof of Theorem~\ref{sv-thm}}\label{sv-proof-sect}

Fix a strict partition $\lambda$.
We derive Theorem~\ref{sv-thm} after  proving a pair of technical lemmas.

\begin{lemma} \label{e-op-SetDTab}
Fix $i \in [n-1]$ and $T \in \SetDTab_n(\lambda)$.
Suppose there are $i$-unpaired letters equal to $i+1$ in $\revrow(T)$.
Let $(x,y)$ be the box of $T$ containing the first such $i+1$. Assume that changing this $i+1$ to $i$ does not yield a set-valued decomposition tableau.
Then:

\ben 

\item[(1)] None of the following occurs:
\ben
\item $x=y$ and some  $z>y$ has $i \in T_{x+1,z}$.
\item $x<y$ and some $z>y$ has $\max(T_{x+1,y}) \geq i \in T_{x+1,z}$.
\item $1<x<y$ and some $1<z < y$ has $\min(T_{x-1,z}) \leq i \in T_{xz}$.
\een
Equivalently, $T$ does not contain any of the following configurations:
\[
 \ytabc{0.8cm}{
 \none & \cdots & i   \\
 T_{xy} & \cdots & \   
 },
\quad
 \ytabc{0.8cm}{
 \geq i & \cdots & i  \\
 T_{xy} & \cdots & \ 
 },
\quord
 \ytabc{0.8cm}{
 i & \cdots & T_{xy}  \\
 \leq i & \cdots & \ 
 }.
 \]
 
 \item[(2)]  A box of $T$  preceding $(x,y)$ in the reverse row reading word order
contains both  $i$ and $i+1$. Let $(a,b)$ be the last such box.
Then $a \in \{x-1,x\}$ and one of the following holds:
\ben
\item $a=x$, $b=y+1$, and $\min(T_{x, y+2})\leq i$, so  $T$ has the configuration
\[
 \ytabc{1.2cm}{
   \cdots & \barr{c} i+1\\ \in T_{xy}  \earr & \barr{c} i,i+1 \\ \in T_{ab}\earr  & \leq i & \cdots}.
   \]
\item  $a =x$, $b>y$, and $\min(T_{x+1,b}) < i$,  so  $T$ has the configuration
\[
  \ytabc{1.2cm}{
 \none & \cdots  &   & \cdots & <i  \\
\ & \cdots &\barr{c} i+1\\ \in T_{xy}  \earr  & \cdots & \barr{c} i,i+1 \\ \in T_{ab}\earr 
 }.
 \]
\item $a=x-1$, $b<y$, and $\max(T_{ay})>i+1$, so  $T$ has the configuration
 \[ \ytabc{1.2cm}{
 \none & \cdots  &   & \cdots & \barr{c} i+1\\ \in T_{xy}  \earr   \\
\ & \cdots &\barr{c} i,i+1 \\ \in T_{ab}\earr  & \cdots & >i+1 
 }.\]
\een

\item[(3)] No box of $T$ between $(a,b)$ and $(x,y)$ in the reverse row reading word order 
contains $i$ or $i+1$.

\item[(4)] Removing $i+1$ from $T_{ab}$ and adding $i$ to $T_{xy}$ 
yields a set-valued decomposition tableau.

\een
\end{lemma}

\begin{proof}
Let $\text{distr}_{i, max}(T)$ be the distribution of $T$ formed by choosing $i$ in every box that contains $\{i, i+1\}$ appearing before $(x,y)$ in $\revrow(T)$, choosing $i+1$ in every box containing $\{i, i+1\}$ that appears after $(x,y)$ in $\revrow(T)$, choosing either $i$ or $i+1$ in boxes that contain $i$ or $i+1$, respectively, but not both, and choosing the largest entry in all other boxes. 
Then let $\text{distr}_{i, min}(T)$ be the distribution of $T$ formed in almost the same way, except that we choose the smallest entry in boxes that do not contain $i$ or $i+1$. For both of these distributions, the  $i+1$ in box $(x,y)$ remains the first $i$-unpaired letter in the reverse row reading word.
Also, both   $\text{distr}_{i, max}(T)$ and $ \text{distr}_{i, min}(T)$
must be valid decomposition tableaux, as are any nonzero tableau obtained by applying a crystal operator $e_j$ or $f_j$ to  $\text{distr}_{i, max}(T)$ or $ \text{distr}_{i, min}(T)$.

Now we can use the operations $\text{distr}_{i, max}$ and $ \text{distr}_{i, min}$ to prove part (1). Suppose the configuration in (1)(a) occurs. 
Then either $i+1 \in T_{x+1,z}$, in which case the $i+1$ in both boxes $T_{xy}$ and $T_{x+1,z}$ create a forbidden pattern from Lemma~\ref{decomptab-characterisation}(a), 
or $i+1 \notin T_{x+1, z}$. In the second case, $e_i(\text{distr}_{i, max}(T))$ is formed from $\text{distr}_{i, max}(T)$ by changing the $i+1$ in box $(x,y)$ to $i$. Then the $i$ in both boxes $(x,y)$ and $(x+1, z)$ of $e_i(\text{distr}_{i, max}(T))$
form a forbidden pattern from Lemma~\ref{decomptab-characterisation}(a). 
Thus, the configuration  in (1)(a) must not occur.

Similarly suppose the configuration in (1)(b) occurs. Then $\text{distr}_{i, max}(T)$ selects $i+1$ from $T_{xy}$, $i$ from $T_{x+1,z}$ and either $\max(T_{x+1,y})$ or $i+1$ or $i$ from $T_{x+1,y}$. In this case, $e_i(\text{distr}_i(T))$ is formed from $\text{distr}_i(T)$ by changing the $i+1$ in box $(x,y)$ to $i$. Now boxes $(x+1,y)$, $(x+1,z)$ and $(x,y)$ form a forbidden pattern in Lemma~\ref{decomptab-characterisation}(a), which is impossible.

Finally suppose the configuration in $(c)$ occurs. Either $i+1 \in T_{xz}$, in which case the minimum entry in box $T_{x-1,z}$ and the $i+1$ in both boxes $T_{xy}$ and $T_{xz}$ create a forbidden pattern from Lemma~\ref{decomptab-characterisation}(a), 
or $i+1 \notin T_{xz}$. In the second case $e_i(\text{distr}_{i, min}(T))$ is formed from $\text{distr}_{i, min}(T)$ by changing the $i+1$ in box $(x,y)$ to $i$. But then boxes $(x,z)$, $(x,y)$ and $(x-1, z)$ in  $e_i(\text{distr}_{i, min}(T))$ form a forbidden pattern in Lemma~\ref{decomptab-characterisation}(a), which is impossible.

We conclude that none of the configurations listed in part (1) occurs.

For part (2),  we observe that changing $i+1$ to $i$ in $T_{xy}$ fails to produce a set-valued decomposition tableau when either (i) the row $x$ is not a hook word for some distribution of $T$, or (ii) the row $x$ is part of a forbidden pattern from Lemma~\ref{decomptab-characterisation}, which may involve either row $x+1$ or $x-1$. 

For case (i), there are two possibilities to consider: ($\alpha$) the $i+1$ in box $(x,y)$ belongs to the decreasing part of row $x$ in a particular distribution $T$, and there exists an $i+1$ in box $(x,y+1)$ or ($\beta$) the $i+1$ in box $(x,y)$ is part of the increasing part of row $x$ in a distribution $T$, and there exists an $i$ in box $(x,y-1)$ that is also part of the increasing part of row $x$ in the same distribution. Possibility ($\alpha$) occurs when there is also an $i$ in box $(x,y+1)$; otherwise, the $i+1$ in box $(x,y+1)$ would be an earlier unpaired $i+1$, and this possibility corresponds to case (2)(a).

On the other hand, case ($\beta$) cannot actually occur. If such a scenario were to arise, then $i+1$ would also be present in box $(x,y-1)$; otherwise, the $i$ in box $(x,y-1)$ would be paired with the $i+1$ in box $(x,y)$. However, this would lead to the contradiction that an element in box $(x,y-2)$ is smaller than $i$, implying that there exists a distribution that has consecutive entries $a, i+1, i+1$ with $a < i$ in row $x$, which violates the condition of each row being a hook word in a decomposition tableau.

Now we examine case (ii). 
As changing $i+1$ to $i$ in $T_{xy}$ does not result in a set-valued decomposition tableau, there must be an earlier unpaired $i+1$ among the $i$-unpaired letters in the reverse reading word of a specific distribution of $T$. Suppose the last such $i+1$ 
occurs in box $(c,d)$. 
In view of part (1),
when changing $i+1$ to $i$ in $T_{xy}$, only the forbidden patterns described in Lemma~\ref{decomptab-characterisation}(b) remain as possibilities. These patterns occur if and only if one of the configurations in part (2)(b) or (2)(c) arises. 

Next we show that the box $T_{cd}$ must contain both $i$ and $i+1$.
Observe that $T_{cd}$ must contain $i+1$ to ensure that some forbidden pattern arises when changing the $i+1$ in $T_{xy}$ to $i$.
Suppose $T_{cd}$ does not contain $i$. Since $T_{xy}$ represents the first occurrence of an $i$-unpaired $i+1$ in $\revrow(T)$, the $i+1$ in box $T_{cd}$ must be paired with an $i$ in a box $T_{xf}$ where $y<f<d$ for case $(2)(b)$, or in a box $T_{ef}$ where $e=x$, $f>y$, or $e=x-1$, $f<d$ for case (2)(c). However, in case (2)(b), the distribution with entries $\min(T_{x+1,b}) < i$, $T_{xf}$, and $T_{cd}$ forms the same type of forbidden pattern right from the start. This is illustrated by the picture below
\[
  \ytabc{1.2cm}{
 \none & \cdots  &   & \cdots &  & \cdots & \textcolor{red}{<i}  \\
\ & \cdots &\barr{c} i + 1 \\ \in T_{xy}\earr & \cdots & \textcolor{red}{i \in T_{xf}} & \cdots & \textcolor{red}{\barr{c} i+1 \\ \in T_{cd} \earr}
 }
 \]
 where the forbidden pattern is highlighted in red.

In the subcase for (2)(c) where $i$ is present in $T_{x-1,f}$ with $f<d$, we consider two slightly different situations: one where $d=y-1$ and another where $d<y-1$. In both situations, we have $\max(T_{xd}) < i+1$, since if $\max(T_{xd}) \geq i+1 $, it would lead to a forbidden pattern with boxes $T_{xd}$, $T_{xy}$, and $T_{cd} = T_{x-1,d}$, as described in Lemma~\ref{decomptab-characterisation}(a).

Now, we focus on the situation where $d=y-1$. In this case, $\max(T_{xd})$ cannot be $i$, because if it were, the $i+1$ in $T_{xy}$ would be paired with this $i$. However, in this situation, the boxes $T_{x,y-1}$, $T_{x-1,f}$, and $T_{x-1,y-1}$ form a forbidden pattern as described in Lemma~\ref{decomptab-characterisation}(b). The situation is illustrated below:

 \[ \ytabc{1.5cm}{
 \none & \cdots  &   & \cdots & \textcolor{red}{<i} & \barr{c} i + 1 \\ \in T_{xy}\earr &  \cdots  \\
\ & \cdots & \textcolor{red}{i} & \cdots & \textcolor{red}{\barr{c} i+1 \\ \in T_{cd} \earr} & >i+1&  \cdots
 }.\]

If $d < y-1$, then $\max(T_{xd})$ still cannot be equal to $i$. If we assume that $\max(T_{xd}) = i$, then there would be an $i+1$ in a box $T_{xg}$ where $d<g<y$. However, it is impossible to have a subword $i, i+1, i+1$ in the boxes $T_{xd}$, $T_{xg}$, $T_{xy}$ because the $i+1 \in T_{xg}$ would lie in both the increasing and decreasing parts of a hook word in a specific distribution of $T$.
Therefore, $\max(T_{xd})$ must be strictly less than $i$, but then the boxes $T_{xd}$, $T_{x-1,f}$, and $T_{x-1,d}$ form a forbidden pattern described in Lemma~\ref{decomptab-characterisation}(b). The situation is illustrated below:
\[ \ytabc{1.5cm}{
 \none & \cdots  &   & \cdots & \textcolor{red}{<i} & \cdots & \barr{c} i + 1 \\ \in T_{xy}\earr &  \cdots  \\
\ & \cdots & \textcolor{red}{i} & \cdots &\textcolor{red}{\barr{c} i + 1 \\ \in T_{cd}\earr}&\cdots & >i+1&  \cdots
 }.\]

Finally, in the subcase for (2)(c) where $i$ is present in $T_{xf}$ with $f>y$, note that our assumptions imply that the $i+1$ in $T_{xy}$ must be part of the decreasing part of row $x$ in any distribution where $i$ is in $T_{xf}$. Consequently, we have $\max(T_{c+1,d}) \geq i+1$. However, this leads to a forbidden pattern of type (a) in Lemma~\ref{decomptab-characterisation} with the boxes $T_{c+1,d}$, $T_{xy}$, and $T_{cd}$. The situation is illustrated below:
 \[ \ytabc{1.2cm}{
 \none & \cdots  & \textcolor{red}{\geq i+1}  & \cdots & \textcolor{red}{\barr{c} i + 1 \\ \in T_{xy}\earr} & \cdots & i & \cdots  \\
\ & \cdots &\textcolor{red}{\barr{c} i,i + 1 \\ \in T_{cd}\earr}& \cdots & >i+1 & \cdots & & \cdots
 }.\]
Therefore, we can conclude that the box $T_{cd}$ must contain both $i$ and $i+1$. 

We will now show that if there exists a box containing both $i$ and $i+1$ located strictly between the box marked $(c,d)$ and $(x,y)$ in the reverse row reading order, for cases (2)(b) and (2)(c), then this box must also yield a forbidden pattern in cases (2)(b) or (2)(c). Consequently, without loss of generality, we can deduce that box $(c,d)$ coincides with $(a,b)$, which is defined as the last such box preceding $(x,y)$ in the reverse row reading order containing both $i$ and $i+1$.

In fact, since $i+1 \in T_{cd}$, the same argument as above demonstrates that no box in $T$ located strictly between $(c,d)$ and $(x,y)$ in the reverse row reading word order can contain $i$. This also implies that no box in $T$ located strictly between $(c,d)$ and $(x,y)$ in the reverse row reading word order can contain $i+1$, since if it did, this $i+1$ would be an earlier $i$-unpaired letter in $\revrow(T)$. Hence, we can conclude that (2) is proven by asserting $(a,b) = (c,d)$, and simultaneously, part (3) is also established. Note that (3) is trivial in case (2)(a).

Finally, part (4) can be concluded in two steps. First, removing $i+1$ from $T_{ab}$ will certainly result in a set-valued decomposition tableau. This is because removing $i+1$ does not introduce new forbidden patterns, and we are only eliminating a subset of the distributions that form decomposition tableaux in $T$. Second, using part (3), we can observe that adding $i$ to $T_{xy}$ creates new forbidden patterns of the form (2)(b) and (2)(c) only under the occurrence of one of the following two cases:
\[
  \ytabc{1.2cm}{
 \none & \cdots  &   & \cdots & <i & \textcolor{red}{<i} \\
\ & \cdots & \barr{c} i, i+1  \\ \in T_{xy}\earr& \cdots & \textcolor{red}{i\in T_{ab}} & \textcolor{red}{i+1}
 }
\]
or
 \[
 \ytabc{1.2cm}{
 \none & \cdots  &  & && \cdots & \textcolor{red}{\barr{c} i,i+1 \\ \in T_{xy}\earr}  \\
\ & \cdots &i\in T_{ab} &i+1 &\textcolor{red}{>i+1}& \cdots & \textcolor{red}{\gg i+1} 
 }.\]
 Here we use the notation ``$\gg i+1$'' to mean that the maximum entry in that box is greater than the maximum entry in the box labeled $>i+1$.
However, these cases lead to forbidden patterns in distributions of $T$ right from the start, as highlighted in red in the above tableaux. Therefore, we can conclude that removing $i+1$ from $T_{ab}$ and adding $i$ to $T_{xy}$ indeed results in a set-valued decomposition tableau.
\end{proof}

Our second technical lemma is the $f$-version of Lemma~\ref{e-op-SetDTab}.

\begin{lemma}\label{f-op-SetDTab}
Fix $i \in [n-1]$ and $T \in \SetDTab_n(\lambda)$.
Suppose there are $i$-unpaired letters equal to $i$ in $\revrow(T)$.
Let $(x,y)$ be the box of $T$ containing 
the last such $i$. Assume that changing this $i$ to $i+1$ does
not yield a set-valued decomposition tableau.
Then:
\ben 

\item[(1)] None of the following occurs:
\ben
\item $1<x\leq y$ and $i+1 \in T_{x-1,x-1}$.
\item $1< x\leq y$ and some $z>y$ has $\min(T_{x-1,y}) \leq i +1\in T_{xz}$.
\item $1<x<y$ and some $1<z < y$ has $\max(T_{xz}) \geq i+1 \in T_{x-1,z}$.
\een
Equivalently, $T$ does not contain any of the following configurations:
\[
 \ytabc{1.2cm}{
 \none & \cdots & T_{xy}   \\
i+1& \cdots & \   
 },
\quad
 \ytabc{1.2cm}{
T_{xy} & \cdots & i+1  \\
\leq i+1 & \cdots & \ 
 },
\quord
 \ytabc{1.2cm}{
\geq i+1 & \cdots & T_{xy}  \\
  i+1 & \cdots & \ 
 }.
 \]
 
 \item[(2)]  A box of $T$ following $(x,y)$ in the reverse row reading word order
contains both  $i$ and $i+1$. Let $(a,b)$ be the first such box.
Then $a \in \{x,x+1\}$ and one of the following holds:
\ben
\item $a=x$, $b=y-1$, and $\min(T_{x, y+1})\leq i$, so $T$ has the configuration
\[
 \ytabc{1.2cm}{
   \cdots &\barr{c} i,i+1  \\ \in T_{ab}\earr & \barr{c} i  \\ \in T_{xy}\earr& \leq i & \cdots}.
   \]
\item $a =x$, $b<y$, and $\min(T_{x+1,y})<i$,  so  $T$ has the configuration
\[
  \ytabc{1.2cm}{
 \none & \cdots  &   & \cdots & <i  \\
\ & \cdots & \barr{c} i,i+1  \\ \in T_{ab}\earr & \cdots & \barr{c} i  \\ \in T_{xy}\earr
 }.
 \]
\item $a=x+1$, $b>y$, and $\max(T_{xb})>i+1$, so  $T$ has the configuration
 \[
   \ytabc{1.2cm}{
 \none & \cdots  &   & \cdots &\barr{c} i,i+1  \\ \in T_{ab}\earr  \\
\ & \cdots & \barr{c} i  \\ \in T_{xy}\earr & \cdots & >i+1 
 }
 .\]
\een

\item[(3)] No box of $T$ between $(x,y)$ and $(a,b)$ in the reverse row reading word order 
contains $i$ or $i+1$.

\item[(4)] Removing $i$ from $T_{ab}$ and adding $i+1$ to $T_{xy}$ 
yields a set-valued decomposition tableau.

\een
\end{lemma}

\begin{proof}
Define $\text{distr}_{i, max}(T)$ and $\text{distr}_{i, min}(T)$ in the same way as in Lemma~\ref{e-op-SetDTab}, but now using the definition of $(x,y)$ as the box of $T$ containing 
the last $i$-unpaired letter equal to $i$ in $\revrow(T)$. This lets us select two specific distribution for $T$ for which the $i$ in $(x,y)$ remains the last $i$-unpaired letter in the reverse row reading orders of both $\revrow(\text{distr}_{i, max}(T))$ and $\revrow(\text{distr}_{i, min}(T))$.
As in the previous proof,
both   $\text{distr}_{i, max}(T)$ and $ \text{distr}_{i, min}(T)$
must be valid decomposition tableaux, as are any nonzero tableau obtained by applying a crystal operator $e_j$ or $f_j$ to  $\text{distr}_{i, max}(T)$ or $ \text{distr}_{i, min}(T)$.

We use the operations $\text{distr}_{i, max}$ and $ \text{distr}_{i, min}$
like in the previous proof to derive part (1). Suppose the configuration in part (a) occurs. 
Then either $i \in T_{x-1,x-1}$, in which case the $i$ entries in boxes $T_{xy}$ and $T_{x-1,x-1}$ create a forbidden pattern from Lemma~\ref{decomptab-characterisation}(a), 
or $i \notin T_{x-1,x-1}$. In the latter case, $f_i(\text{distr}_{i, max}(T))$ is formed from $\text{distr}_{i, max}(T)$ by changing the $i$ in box $(x,y)$ to $i+1$. But then the $i+1$ entries in boxes $(x,y)$ and $(x-1,x-1)$ of $f_i(\text{distr}_{i, max}(T))$
form a forbidden pattern from Lemma~\ref{decomptab-characterisation}(a), which is impossible.

Similarly suppose the configuration in (b) occurs. Then $\text{distr}_{i, min}(T)$ selects $i$ from $T_{xy}$, $i+1$ from $T_{xz}$ and $\min(T_{x-1,y})$, $i$ or $i+1$ from $T_{x-1,y}$. In this case, $f_i(\text{distr}_i(T))$ is formed from $\text{distr}_i(T)$ by changing the $i$ in box $(x,y)$ to $i+1$. Now boxes $(x+1,y)$, $(x+1,z)$ and $(x,y)$  of $f_i(\text{distr}_i(T))$ form a forbidden pattern in Lemma~\ref{decomptab-characterisation}(a), which is impossible.

Finally suppose the configuration in (c) occurs. Either $i \in T_{x-1,z}$, in which case the maximum entry in box $T_{xz}$ and the $i$ in both boxes $T_{xy}$ and $T_{x-1,z}$ create a forbidden pattern from Lemma~\ref{decomptab-characterisation}(a), 
or $i \notin T_{x-1,z}$. In the latter case $f_i(\text{distr}_{i, max}(T))$ is formed from $\text{distr}_{i, max}(T)$ by changing the $i$ in box $(x,y)$ to $i+1$. Now boxes $(x,z)$, $(x,y)$ and $(x-1, z)$ of $f_i(\text{distr}_{i, max}(T))$ form a forbidden pattern from Lemma~\ref{decomptab-characterisation}(a), which is impossible.

We conclude that none of the configurations listed in part (1) occurs.

For part (2),  note that  changing $i$ to $i+1$ in $T_{xy}$ fails to produce a set-valued decomposition tableau  when either (i) the row $x$ is not a hook word for some distribution of $T$, or (ii) the row $x$ is part of a forbidden pattern as described in Lemma~\ref{decomptab-characterisation}, which may involve either row $x+1$ or $x-1$. 

For case (i), there are two possibilities to consider: ($\alpha$) the $i$ in box $(x,y)$ belongs to the decreasing part of row $x$ in a particular distribution $T$, and there exists an $i$ in box $(x,y-1)$ or ($\beta$) the $i$ in box $(x,y)$ is part of the increasing part of row $x$ in a distribution $T$, and there exists an $i+1$ in box $(x,y+1)$. Possibility ($\alpha$) occurs when there is also an $i+1$ in box $(x,y-1)$; otherwise, the $i$ in box $(x,y+1)$ would be a later unpaired $i$, and this possibility corresponds to case (2)(a).

On the other hand, possibility ($\beta$) cannot actually arise. If such a scenario were to occur, it would mean that $i$ would also be present in box $(x,y+1)$; otherwise, the $i$ in box $(x,y)$ would be paired with the $i+1$ in box $(x,y+1)$. However, this would lead to the contradiction that an element in box $(x,y-1)$ is smaller than $i$, implying that there exists a distribution that has consecutive entries $a, i, i$ with $a < i$ in row $x$, which violates the condition of each row being a hook word in a decomposition tableau.

Now we examine case (ii).
As changing $i$ to $i+1$ in $T_{xy}$ does not result in a set-valued decomposition tableau, there must be a later unpaired $i$ among the $i$-unpaired letters in the reverse reading word of a specific distribution of $T$. 
Suppose the first such $i$ occurs in box $(c,d)$.
In view of part (1),
when changing $i$ to $i+1$ in $T_{xy}$, only the forbidden patterns described in Lemma~\ref{decomptab-characterisation}(b) remain as possibilities. These patterns occur if and only if one of the configurations in part (2)(b) or (2)(c) arises. 

Next we show that the box $T_{cd}$ must contain both $i$ and $i+1$.
Observe that $T_{cd}$ must contain $i$ to ensure that some forbidden pattern arises when changing the $i$ in $T_{xy}$ to $i+1$.
Suppose $T_{cd}$ does not contain $i+1$. Since $T_{xy}$ represents the last occurrence of an $i$-unpaired $i$ in $\revrow(T)$, the $i$ in box $T_{cd}$ must be paired with an $i+1$ in a box $T_{xf}$ where $d<f<y$ for case (2)(b), or in a box $T_{ef}$ where $e=x$, $f<y$, or $e=x+1$, $f>d$ for case (2)(c). However, in case (2)(b), the distribution with entries $i \in T_{cd}$, $x+1 \in T_{xf}$, and $i \in T_{xy}$ violates the condition of being a hook word for each row in $T$ right from the start. This is illustrated in the picture below:
\[
  \ytabc{1.5cm}{
\ & \cdots &i \in T_{cd} & \cdots &\barr{c} i +1 \\ \in T_{xf}\earr& \cdots & i \in T_{xy}
 }.
 \]

In the subcase for (2)(c) where $i+1$ is present in $T_{x+1, f}$ with $f>d$, notice that $\max(T_{xf}) > \max(T_{xd})$, since $\max(T_{xd})>i+1 > i \in T_{xy}$. However, the entries $i+1$, $\max(T_{xd})$ and $\max(T_{xf})$ in boxes $T_{x+1, f}$, $T_{xd}$ and $T_{xf}$ respectively form a forbidden pattern of type (b) in Lemma~\ref{decomptab-characterisation}. The situation is illustrated below:
 \[ \ytabc{1.5cm}{
 \none & \cdots  &   & \cdots  & i \in T_{cd} &  \cdots & \textcolor{red}{\barr{c} i+1 \\ \in T_{x+1,f}\earr}  \\
\ & \cdots & i \in T_{xy} & \cdots  & \textcolor{red}{>i+1}&  \cdots & \textcolor{red}{\gg i+1}
 },\] 
where  the notation ``$\gg i+1$'' again means that the maximum entry in that box is greater than the maximum entry in the box labeled $>i+1$.

Finally, in the subcase for (2)(c) where $i+1$ is present in $T_{xf}$ with $f<y$, the entries $i$, $i+1$, and $\max(T_{xd})$ in boxes $T_{cd}$, $T_{xf}$ and $T_{xd}$ respectively form a forbidden pattern of type (b) in Lemma~\ref{decomptab-characterisation}
The situation is illustrated below:
 \[ \ytabc{1.2cm}{
 \none & \cdots  &   & \cdots &  & \cdots &\textcolor{red}{ i \in T_{cd} }& \cdots  \\
\ & \cdots &\textcolor{red}{\barr{c}  i+1 \\ \in T_{xf}\earr} & \cdots & T_{xy} & \cdots & \textcolor{red}{>i+1} & \cdots
 }.\]
Therefore, we conclude that the box $T_{cd}$ must contain both $i$ and $i+1$. 

We will now show that if there exists a box containing both $i$ and $i+1$ located strictly between the box marked $(x,y)$ and $(c,d)$ in the reverse row reading order, for cases (2)(b) and (2)(c), then this box must also yield a forbidden pattern in cases (2)(b) or (2)(c). Consequently, without loss of generality, we can deduce that box $(c,d)$ coincides with $(a,b)$, which is defined as the first box after $(x,y)$ in the reverse row reading order containing both $i$ and $i+1$.

In fact, since $i \in T_{cd}$, the same argument as above demonstrates that no box in $T$ located strictly between $(x,y)$ and $(c,d)$ in the reverse row reading word order can contain $i+1$. This also implies that no box in $T$ located strictly between $(x,y)$ and $(c,d)$ in the reverse row reading word order can contain $i$, since if it did, this $i$ would be a later $i$-unpaired letter in $\revrow(T)$. Hence, we can conclude that (2) is proven by asserting $(a,b) = (c,d)$, and simultaneously, part (3) is also established. Note that (3) is trivial in case (2)(a).

Finally, part (4) can be concluded in two steps. First, removing $i$ from $T_{ab}$ will certainly result in a set-valued decomposition tableau.
  Second, using part (3), we observe that adding $i+1$ to $T_{xy}$ creates new forbidden patterns of the form (2)(b) and (2)(c) only when one of the following cases occurs:
\[
  \ytabc{1.5cm}{
 \none & \cdots  &   & \cdots & <i &\cdots & \textcolor{red}{i} \\
\ & \cdots &\textcolor{red}{\barr{c} i+1  \\ \in T_{ab} \earr}& \cdots &\barr{c} i, i+1  \\ \in T_{xy}\earr &\cdots &  \textcolor{red}{>i+1}
 }
\]
or
 \[
 \ytabc{1.5cm}{
 \none & \cdots  &  & \cdots & i & \cdots & \textcolor{red}{\barr{c} i+1 \\ \in T_{ab} \earr}  \\
\ & \cdots & \barr{c} i, i+1 \\ \in T_{xy} \earr & \cdots &\textcolor{red}{>i+1} & \cdots &  \textcolor{red}{\gg i+1} 
 }.\]
However, these cases lead to forbidden patterns in distributions of $T$ right from the start, as highlighted in red in the above tableaux. Therefore, we can conclude that removing $i$ from $T_{ab}$ and adding $i+1$ to $T_{xy}$ indeed results in a set-valued decomposition tableau.
\end{proof}

 We can now supply the main proof of this section.

\begin{proof}[Proof of Theorem~\ref{sv-thm}]
It is clear from Lemmas~\ref{e-op-SetDTab} and \ref{f-op-SetDTab}---in particular, part
(4) of each statement---that if $i \in \{\bar 1,1,2,\dots,n-1\}$ then our definitions of $e_i$ and $f_i$
give well-defined maps $\SetDTab_n(\lambda) \to \SetDTab_n(\lambda)\sqcup\{0\}$.

It is also evident that 
when they do not act as a zero, $e_i$ and $f_i$ add or subtract $\e_i-\e_{i+1}$ to the weight,
while $e_{\bar 1}$ and $f_{\bar 1}$ add or subtract $\e_1-\e_{2}$.
For $i \in[n-1]$ and $T\in \SetDTab_n(\lambda)$, we have $\varphi_i(T) - \varepsilon_i(T)
=\weight(T)_i-\weight(T)_{i+1}$
since $\varphi_i(T)$ and $\varepsilon_i(T)$ are the respective numbers of $i$-unpaired  
letters in $\revrow(T)$ equal to $i$ and $i+1$.

The remarks before Theorem~\ref{sv-thm} show
that if $T,U \in \SetDTab_n(\lambda)$ then $e_{\bar 1}(T) =U$ if and only if $T = f_{\bar 1}(U)$.
The only thing remaining is the analogous property for $e_i$ and $f_i$ when $i \in[n-1]$,
which we can check using 
Lemmas~\ref{e-op-SetDTab} and \ref{f-op-SetDTab}.

Suppose $e_i(T) \neq 0$. If $e_i(T)$ is formed from $T$
by changing the first $i$-unpaired $i+1$ letter in $\revrow(T)$ to $i$,
then this newly created $i$ will be the last $i$-unpaired $i$ in $\revrow(e_i(T))$
and so $f_i$ will act to change it back to $i+1$, giving $f_i(e_i(T)) = T$. 

If
instead $e_i(T)$ is formed by removing $i+1$ from $T_{ab}$ and adding $i$ to $T_{xy}$
using the notation in Lemma~\ref{e-op-SetDTab}, then the remaining $i \in T_{ab}$
will become the last $i$-unpaired $i$ in $\revrow(e_i(T))$.
Inspecting cases (a), (b), and (c) in part (2) of Lemma~\ref{e-op-SetDTab} shows that changing this $i$ to $i+1$ will not produce a set-valued decomposition tableau.
We are therefore in the situation of Lemma~\ref{f-op-SetDTab},
from which it follows again that $f_i(e_i(T)) = T$.
The argument to show that $e_i(f_i(T)) = T$ when $f_i(T) \neq 0$ is similar, just
swapping the roles of Lemmas~\ref{e-op-SetDTab} and \ref{f-op-SetDTab}.
\end{proof}

\section{Second construction}

This section describes  another crystal structure on set-valued decomposition tableaux,
which is based on a construction of Yu \cite{Yu} for unshifted tableaux.
Unlike our results in the previous section, we can place this crystal in a more general
abstract framework, involving objects that we call \defn{square root crystals}. 
We will see that a natural tensor product for square root crystals  motivates the crystal operator definitions in \cite{Yu}, and also provides a novel interpretation
of the nonnegative integer coefficients appearing in the product expansion $G_\lambda G_\mu = \sum_\nu c_{\lambda\mu}^\nu G_\nu$.

\subsection{Abstract square root crystals}\label{sqrt-crystal-sect}

Let $n$ be any positive integer.
Suppose $\cB$ is a nonempty set with maps $\weight  :  \cB\to \ZZ^n$
and 
$e'_i,f'_i  :  \cB \to \cB \sqcup \{0\}$  for $i \in [n-1]$,
where $0 \notin \cB$.
Define $\varepsilon'_i, \varphi'_i  :  \cB \to \{0,1,2,\dots\}\sqcup \{\infty\}$ 
by 
\be
\label{var'-eq}
\varepsilon'_i(b) := \sup\left\{ k\geq 0 \mid (e'_i)^k(b) \neq 0\right\}
\text{ and }
\varphi'_i(b) := \sup\left\{ k \geq 0: (f'_i)^k(b) \neq 0\right\}.
\ee

\begin{definition}\label{sqgln-def}
The set $\cB$ is a \defn{$\sqgln$-crystal}
if for all $i \in[n-1]$
and $b,c \in \cB$ we have both:
\ben 
\item[(a)] $\varepsilon'_i(b) + \varphi'_i(b) \in 2\NN$ is even with 
$ \frac{\varphi'_i(b)  - \varepsilon'_i(b)}{2} = \weight(b)_i-\weight(b)_{i+1},$ and 
\item[(b)]   $e'_i(b) = c$ if and only if $b = f'_i(c)$,
in which case
$ \weight(c) - \weight(b) = \begin{cases} \e_i &\text{if $\varepsilon'_i(b)$ is even} \\ -\e_{i+1}&\text{if $\varepsilon'_i(b)$ is odd}.
\end{cases}
$
\een
\end{definition}

The ``queer'' extension of this definition goes as follows.
Suppose $\cB$ is a $\sqgln$-crystal 
with additional maps $e'_{\bar 1},f'_{\bar 1} : \cB \to \cB \sqcup \{0\}$.
Define $\varepsilon'_{ \bar 1}$ and $\varphi'_{ \bar 1} $ 
as in \eqref{var'-eq} with $i={\bar 1}$.

\begin{definition}\label{sqqn-def}
When $n\geq 2$, we define
  $\cB$
to be a \defn{$\sqqn$-crystal} 
if  $e'_{\bar 1}$ and $f'_{\bar 1}$ commute with $e'_i$ and $f'_i$ while preserving $\varepsilon'_i$ and $\varphi'_i$ for all $3\leq i \leq n-1$,
such that for all $b,c \in \cB$ we have both:
\ben

\item[(a)] $\weight(b) \in \NN^n$ and $\varepsilon'_{\bar 1}(b) + \varphi'_{\bar 1}(b) =\begin{cases}
0 &\text{if }\weight(b)_1 = \weight(b)_2 = 0 \\
2 &\text{otherwise};
\end{cases}$ 

\item[(b)]   $e'_{\bar 1}(b)=c $ if and only if $b=f'_{\bar 1}(c) $,
in which case 
$\weight(c) - \weight(b) = \begin{cases} \e_1 &\text{if $\varepsilon'_{\bar 1}(b)=2$} \\ -\e_{2}&\text{if $\varepsilon'_{\bar 1}(b)=1$}.
\end{cases}$
\een
When $n=1$, we define a \defn{$\sqrt{\q_1}$-crystal} to be the same thing 
as a $\sqrt{\gl_1}$-crystal.
\end{definition}

We associate to each $\sqgln$- and $\sqqn$-crystal a \defn{crystal graph} in the usual way,
with arrows $b\xrightarrow{i}c$ whenever $f'_i(b) = c\neq 0$.
Unlike the classical case, these edge-labeled directed graphs are no longer acyclic in general.
Two $\sqgln$- or $\sqqn$-crystal  are \defn{isomorphic} if there is a weight-preserving graph 
isomorphism between their crystal graphs.

\begin{example}\label{SS-ex}
Let $\SS_n$ be the set of nonempty subsets of $[n]$.
For $i \in [n-1]$ and $S \in \SS_n$, define 
\[ \ba e'_i(S) &:= \begin{cases}
S\sqcup \{i\}&\text{if }S\cap \{i,i+1\} = \{i+1\}, \\
S\setminus \{i+1\}&\text{if }S\cap \{i,i+1\} = \{i,i+1\}, \\
0&\text{otherwise},
\end{cases}
\\
f'_i(S) &:= \begin{cases}
S\sqcup \{i+1\}&\text{if }S\cap \{i,i+1\} = \{i\}, \\
S\setminus \{i\}&\text{if }S\cap \{i,i+1\} = \{i,i+1\}, \\
0&\text{otherwise}.
\end{cases}
\ea
\]
Also  let $\weight(S) = \sum_{i \in S} \e_i \in \NN^n$. 
Relative to these maps, the set $\SS_n$ is a $\sqgln$--crystal,
which we refer to as the \defn{standard crystal}.
Its crystal graph for some small values of $n$ is shown in Figure~\ref{SS-fig}.
When $n\geq 2$, define $e'_{\bar 1}(S)  = e'_1(S)$ and $f'_{\bar 1}(S) = f'_1(S)$.
Then $\SS_n$ is also a $\sqqn$-crystal.
 
\end{example}

\begin{figure}
\centerline{\ytableausetup{boxsize = .35cm,aligntableaux=center}

\begin{tikzpicture}[>=latex,line join=bevel,]
  \pgfsetlinewidth{1bp}
\tiny%
\pgfsetcolor{black}
  \pgfsetcolor{red}
  \draw [->,solid] (68.256bp,23.426bp) .. controls (65.678bp,27.073bp) and (64.834bp,34.702bp)  .. (68.779bp,51.014bp);
  \pgfsetcolor{blue}
  \draw [->,solid] (17.811bp,8.8026bp) .. controls (20.735bp,9.7641bp) and (24.897bp,11.133bp)  .. (38.593bp,15.636bp);
  \pgfsetcolor{teal}
  \draw [->,solid] (75.934bp,50.918bp) .. controls (78.512bp,47.27bp) and (79.356bp,39.641bp)  .. (75.411bp,23.33bp);
  \pgfsetcolor{red}
  \draw [->,solid] (76.191bp,53.826bp) .. controls (78.707bp,54.325bp) and (82.213bp,55.021bp)  .. (95.86bp,57.729bp);
  \draw [->,solid] (29.116bp,71.813bp) .. controls (26.09bp,69.126bp) and (21.579bp,66.655bp)  .. (7.2594bp,63.655bp);
  \pgfsetcolor{teal}
  \draw [->,solid] (34.544bp,71.559bp) .. controls (35.026bp,69.795bp) and (35.556bp,67.853bp)  .. (38.741bp,56.2bp);
  \pgfsetcolor{blue}
  \draw [->,solid] (36.109bp,79.95bp) .. controls (39.694bp,80.554bp) and (45.365bp,81.511bp)  .. (60.924bp,84.134bp);
  \draw [->,solid] (10.976bp,128.6bp) .. controls (11.656bp,127.75bp) and (12.404bp,126.81bp)  .. (19.523bp,117.9bp);
  \draw [->,solid] (6.8052bp,73.609bp) .. controls (9.8309bp,76.296bp) and (14.342bp,78.767bp)  .. (28.662bp,81.767bp);
  \pgfsetcolor{teal}
  \draw [->,solid] (5.5144bp,58.544bp) .. controls (5.8239bp,57.384bp) and (6.154bp,56.147bp)  .. (9.0859bp,45.158bp);
  \pgfsetcolor{blue}
  \draw [->,solid] (44.706bp,49.096bp) .. controls (48.19bp,49.597bp) and (53.525bp,50.363bp)  .. (68.902bp,52.572bp);
  \pgfsetcolor{teal}
  \draw [->,solid] (44.488bp,46.631bp) .. controls (47.252bp,43.245bp) and (48.611bp,35.938bp)  .. (46.052bp,19.522bp);
  \pgfsetcolor{red}
  \draw [->,solid] (37.039bp,41.277bp) .. controls (34.657bp,39.434bp) and (31.46bp,37.769bp)  .. (18.346bp,34.383bp);
  \pgfsetcolor{teal}
  \draw [->,solid] (66.469bp,77.082bp) .. controls (66.97bp,75.059bp) and (67.532bp,72.791bp)  .. (70.535bp,60.673bp);
  \pgfsetcolor{red}
  \draw [->,solid] (68.377bp,85.28bp) .. controls (70.917bp,85.634bp) and (74.456bp,86.127bp)  .. (88.23bp,88.047bp);
  \pgfsetcolor{blue}
  \draw [->,solid] (45.948bp,17.407bp) .. controls (49.038bp,17.869bp) and (53.688bp,18.564bp)  .. (68.134bp,20.723bp);
  \pgfsetcolor{red}
  \draw [->,solid] (38.685bp,18.777bp) .. controls (35.921bp,22.163bp) and (34.562bp,29.469bp)  .. (37.121bp,45.886bp);
  \draw [->,solid] (58.3bp,107.37bp) .. controls (58.74bp,105.78bp) and (59.221bp,104.04bp)  .. (62.483bp,92.275bp);
  \pgfsetcolor{teal}
  \draw [->,solid] (11.844bp,29.824bp) .. controls (11.968bp,28.441bp) and (12.102bp,26.946bp)  .. (13.144bp,15.345bp);
  \pgfsetcolor{blue}
  \draw [->,solid] (15.471bp,45.003bp) .. controls (18.488bp,47.203bp) and (22.631bp,49.116bp)  .. (37.026bp,51.285bp);
  \pgfsetcolor{teal}
  \draw [->,solid] (97.539bp,81.534bp) .. controls (97.889bp,80.151bp) and (98.267bp,78.654bp)  .. (101.21bp,67.038bp);
  \draw [->,solid] (110.41bp,57.449bp) .. controls (112.81bp,56.872bp) and (115.57bp,56.209bp)  .. (128.24bp,53.164bp);
  \pgfsetcolor{red}
  \draw [->,solid] (27.477bp,102.28bp) .. controls (27.824bp,100.67bp) and (28.203bp,98.911bp)  .. (30.776bp,86.967bp);
  \pgfsetcolor{blue}
  \draw [->,solid] (32.853bp,111.18bp) .. controls (34.77bp,111.5bp) and (36.933bp,111.85bp)  .. (49.177bp,113.85bp);
\begin{scope}
  \definecolor{strokecol}{rgb}{0.0,0.0,0.0};
  \pgfsetstrokecolor{strokecol}
  \draw (71.772bp,21.266bp) node {$   {\begin{ytableau}*(white) 24\end{ytableau}}$};
\end{scope}
\begin{scope}
  \definecolor{strokecol}{rgb}{0.0,0.0,0.0};
  \pgfsetstrokecolor{strokecol}
  \draw (72.418bp,53.077bp) node {$   {\begin{ytableau}*(white) \barr{l}23 \\[-2pt] 4 \earr\end{ytableau}}$};
\end{scope}
\begin{scope}
  \definecolor{strokecol}{rgb}{0.0,0.0,0.0};
  \pgfsetstrokecolor{strokecol}
  \draw (13.849bp,7.5bp) node {$   {\begin{ytableau}*(white) 14\end{ytableau}}$};
\end{scope}
\begin{scope}
  \definecolor{strokecol}{rgb}{0.0,0.0,0.0};
  \pgfsetstrokecolor{strokecol}
  \draw (42.339bp,16.868bp) node {$   {\begin{ytableau}*(white) \barr{l}12 \\[-2pt] 4 \earr\end{ytableau}}$};
\end{scope}
\begin{scope}
  \definecolor{strokecol}{rgb}{0.0,0.0,0.0};
  \pgfsetstrokecolor{strokecol}
  \draw (103.19bp,59.183bp) node {$   {\begin{ytableau}*(white) 34\end{ytableau}}$};
\end{scope}
\begin{scope}
  \definecolor{strokecol}{rgb}{0.0,0.0,0.0};
  \pgfsetstrokecolor{strokecol}
  \draw (32.421bp,79.328bp) node {$   {\begin{ytableau}*(white) \barr{l}12 \\[-2pt] 3 \earr\end{ytableau}}$};
\end{scope}
\begin{scope}
  \definecolor{strokecol}{rgb}{0.0,0.0,0.0};
  \pgfsetstrokecolor{strokecol}
  \draw (3.5bp,66.094bp) node {$   {\begin{ytableau}*(white) 13\end{ytableau}}$};
\end{scope}
\begin{scope}
  \definecolor{strokecol}{rgb}{0.0,0.0,0.0};
  \pgfsetstrokecolor{strokecol}
  \draw (40.834bp,48.54bp) node {$   {\begin{ytableau}*(white) \barr{l}12 \\[-2pt] 34 \earr\end{ytableau}}$};
\end{scope}
\begin{scope}
  \definecolor{strokecol}{rgb}{0.0,0.0,0.0};
  \pgfsetstrokecolor{strokecol}
  \draw (64.569bp,84.749bp) node {$   {\begin{ytableau}*(white) 23\end{ytableau}}$};
\end{scope}
\begin{scope}
  \definecolor{strokecol}{rgb}{0.0,0.0,0.0};
  \pgfsetstrokecolor{strokecol}
  \draw (131.8bp,52.309bp) node {$   {\begin{ytableau}*(white) 4\end{ytableau}}$};
\end{scope}
\begin{scope}
  \definecolor{strokecol}{rgb}{0.0,0.0,0.0};
  \pgfsetstrokecolor{strokecol}
  \draw (7.4331bp,133.04bp) node {$   {\begin{ytableau}*(white) 1\end{ytableau}}$};
\end{scope}
\begin{scope}
  \definecolor{strokecol}{rgb}{0.0,0.0,0.0};
  \pgfsetstrokecolor{strokecol}
  \draw (25.808bp,110.03bp) node {$   {\begin{ytableau}*(white) 12\end{ytableau}}$};
\end{scope}
\begin{scope}
  \definecolor{strokecol}{rgb}{0.0,0.0,0.0};
  \pgfsetstrokecolor{strokecol}
  \draw (11.167bp,37.358bp) node {$   {\begin{ytableau}*(white) \barr{l}13 \\[-2pt] 4 \earr\end{ytableau}}$};
\end{scope}
\begin{scope}
  \definecolor{strokecol}{rgb}{0.0,0.0,0.0};
  \pgfsetstrokecolor{strokecol}
  \draw (95.63bp,89.078bp) node {$   {\begin{ytableau}*(white) 3\end{ytableau}}$};
\end{scope}
\begin{scope}
  \definecolor{strokecol}{rgb}{0.0,0.0,0.0};
  \pgfsetstrokecolor{strokecol}
  \draw (56.184bp,115.0bp) node {$   {\begin{ytableau}*(white) 2\end{ytableau}}$};
\end{scope}
\end{tikzpicture}\input{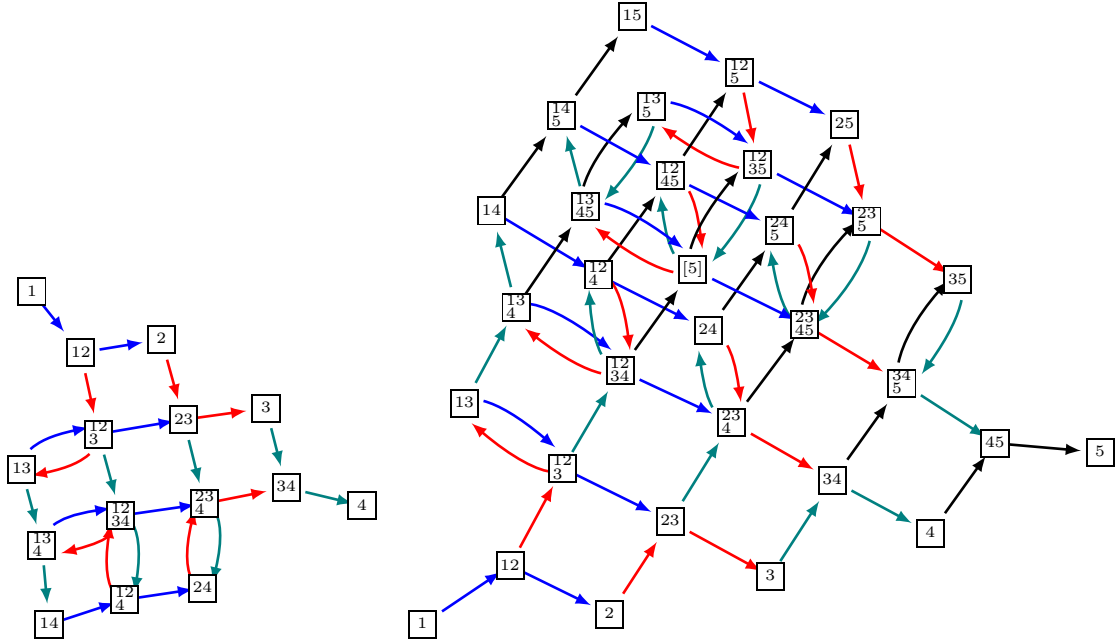}}
\caption{Crystal graphs of the standard $\sqrt{\gl_4}$ and $\sqrt{\gl_5}$-crystals.
We have omitted the edge labels, but have drawn all 
 arrows with a given label using the same color. The omitted edge labels
can be inferred by comparing the weights of each arrow's
source and target.
}
\label{SS-fig}
\end{figure}

We define the \defn{character} of a finite $\sqgln$- or $\sqqn$-crystal in the same way as for ordinary crystals.
The following result is immediate from the definitions, and motivates our terminology:

\begin{proposition}\label{sq-prop}
Any $\sqgln$-crystal (respectively, $\sqqn$-crystal)
is a $\gl_n$-crystal  (respectively, $\q_n$-crystal)
relative to the operators $(e'_i)^2$ and $(f'_i)^2$,
setting $e'_i(0) =f'_i(0)=0$.
This object is seminormal as a $\gl_n$-crystal
and so, when finite, its character is a symmetric polynomial.
\end{proposition}

The ``square'' of a $\sqqn$-crystal 
is not necessarily seminormal as a $\q_n$-crystal since one can have $e_{\bar1}(b) = f_{\bar1}(b)=0$
without having $\weight(b)_1 = \weight(b)_2=0$.
The $\gl_n$- and $\q_n$-crystals derived from $\SS_n$ via Proposition~\ref{sq-prop}
are the standard crystals $\BB_n$ from Examples~\ref{st-ex1} and \ref{st-ex2}.

\subsection{Square root crystals of words}\label{sq-words-sect}

Given $m \in \NN$, let 
$\SVWords_n(m)$ be the set of $m$-element sequences $S=(S_1,S_2,\dots,S_m)$ 
where each $\varnothing \subsetneq S_i \subseteq [n]$, or equivalently $S_i \in \SS_n$.
We refer to elements of $\SVWords_n(m)$ as \defn{set-valued words}.

In \cite{Yu}, Yu identifies a remarkable $\sqgln$-crystal on $\SVWords_n(m)$.
We review this structure below, and indicate
how it can be extended to a $\sqqn$-crystal.
In the next section, we show how this construction leads to 
our second crystal structure on set-valued decomposition tableaux.

For the rest of this part, fix a set-valued word
$S=(S_1,S_2,\dots,S_m)$.
All of the following definitions were introduced in \cite{Yu} in a slightly more restricted context.

\begin{definition} \label{iword-def}
For each $i \in   \PP$,
the \defn{$i$-word} of $S$
is the following word composed of
``$($'', ``$)$'', and ``$-$''
characters concatenated together.
Read through the set-valued entries of $S$ from left to right.
For each entry containing
$i$ but not $i+1$,
we write the single character ``$)$''.
For each entry containing
$i+1$ but not $i$,
we write the single character ``$($''.
Finally, for each entry containing both
$i$ and $i + 1$,
we write the three characters ``$)-($''. 
\end{definition}

The only difference between this definition and \cite[Def.~4.1]{Yu}
is that Yu requires $S$ to be (the column reading word of) a semistandard set-valued tableau, rather than any set-valued word.

\begin{example}
We represent the set-valued word $S=(S_1,S_2,\dots,S_m)$
as the one-column  tableau 
\[S = \ytabc{.6cm}{ S_1 \\ S_2 \\ \vdots \\ S_m},\]
with commas omitted in each entry.
Using this convention, the set-valued word
\[
S = 
\ytabc{.5cm}{
23 \\ 2 \\ 23 \\ 34 \\ 1 \\ 2 \\ 1
}  \in \SVWords_4(7)
\]
has 1-word ``$((()()$'',
2-word ``$)-())-(()$'',
and  3-word ``$)))-($''.
\end{example}

Yu \cite{Yu} divides the characters in the $i$-word into equivalence classes in the following way.
Ignore the ``$-$'' symbols and pair 
the parentheses ``$($'' with ``$)$'' in the usual way. 
Then we require:
\begin{itemize}
\item If a left parenthesis ``$($'' is paired with a right paranthesis ``$)$'',
then these two characters
and everything between them
is in the same equivalence class.

\item For each ``$)-($'', 
these three characters are in the 
same equivalence class.

\end{itemize}
Each of the resulting 
equivalence classes is a contiguous 
sub-word.

\begin{example}
\label{E: forms}
If the $i$-word is
``$))-(())-())-(()-($''
then its distinct equivalence classes are
``${)} $''
and 
``${)-(())-()} $''
and
``${)-(}$''
and
``${()-(}$''.
\end{example}

Observe that 
any unpaired ``$)$'' in the $i$-word
must be the first character in its class,
and any unpaired ``$($'' must be the 
last character in its class. 
As in \cite{Yu}, we distinguish
between classes in the $i$-word using the terminology below:
\begin{enumerate}
\item[(a)] A class is a \defn{null form}
it has no unpaired ``$($'' or ``$)$'' characters.
For example: ``$(()-())-()$''. 

\item[(b)] A class is a \defn{left form}
if has no unpaired ``$)$'' characters but
ends with an unpaired ``$($''. 
\newline
This class is either ``$($'' or ``$(u)-($''
for some word $u$.
For example: ``$(())-($''. 

\item[(c)] A  class is a \defn{right form} if it
has no unpaired ``$($'' characters but
starts with an unpaired ``$)$''.
This class is either ``$)$'' or ``$)-(u)$''
for some word $u$.
For example: ``$)-()-()$''. 

\item[(d)] A class is a \defn{combined form} if it
starts with an unpaired ``$)$''   and
ends with an unpaired ``$($''.
This class is either ``$)-($'' or ``$)-(u)-($''
for some word $u$.
For example: ``$)-()-(())-($''. 
\end{enumerate}

If we ignore the null forms,
then in any $i$-word we have zero or more right forms,
followed by at most one 
combined form, followed by
zero or more left forms.

We may now define $\sqqn$-crystal operators on $\SVWords_n(m)$.
The operators indexed by $i \in [n-1]$ are identical to ones considered in \cite[\S4.1]{Yu},
with the same caveats as after Definition~\ref{iword-def}.
Recall that we let $S = (S_1,S_2,\dots,S_m)$ be an arbitrary element of $\SVWords_n(m)$.

\begin{definition}
For $i \in [n-1]$, construct $e'_i(S) \in\SVWords_n(m)\sqcup \{0\}$  as follows:
\begin{itemize}
\item If the $i$-word of $S$ has a combined form, 
then we find the entry in $S$ that corresponds to 
``$)-($'' at the end of this combined form.
We remove $i+1$ from this entry to obtain $e_i'(S)$.
\item  Otherwise, if the $i$-word of $S$ has no left forms,
then we set $e_i'(S) = 0$.
\item Otherwise, find the first left form
in the $i$-word of $S$, and then 
find the entry in $S$ that corresponds to
``$($'' at the start of this left form.
Add $i$ to this entry to obtain $e_i'(S)$.
\end{itemize}
\end{definition}

\begin{definition}
For $i \in [n-1]$, construct $f'_i(S) \in\SVWords_n(m)\sqcup \{0\}$  as follows:
\begin{itemize}
\item If the $i$-word of $S$ has a combined form, 
then we find the entry in $S$ that corresponds to 
``$)-($'' at the beginning of this combined form.
We remove $i$ from this entry to obtain~$f_i'(S)$.
\item 
Otherwise, if the $i$-word of $S$ has no right forms,  
then we set $f_i'(S) = 0$.
\item Otherwise, find the last right form
in the $i$-word of $S$, and then
find the entry in $S$ that corresponds to
``$)$'' at the end of this right form.
Add $i+1$ to this entry to obtain $f_i'(S)$.
\end{itemize}
\end{definition}
 
\begin{example}
The $f_2'$ operator has this effect on these  set-valued words:
\[
\ytabc{.5cm}{
23 \\
2 \\
2 \\
34 \\ 
1 \\
2 \\
1}
\ \xrightarrow{f'_2}\ 
\ytabc{.5cm}{
23 \\
2 \\
23 \\
34 \\
1 \\
2 \\
1}
\ \xrightarrow{f'_2}\ 
\ytabc{.5cm}{
23 \\
2 \\
3 \\
34 \\
1 \\
2 \\
1}
\ \xrightarrow{f'_2}\ 
\ytabc{.5cm}{
23 \\
23 \\
3 \\
34 \\
1 \\
2 \\
1}
\ \xrightarrow{f'_2}\ 
\ytabc{.5cm}{
3 \\
23 \\
3 \\
34 \\
1 \\
2 \\
1} \ \xrightarrow{f'_2}\ 0.
\]
\end{example}

We include one lemma before proceeding.
Given $S = (S_1,S_2,\dots,S_m) \in \SVWords_n(m)$,
let $w(S)$ be the word formed by replacing each $S_i$ by its entries listed in increasing order.
For example, if $S=(\{4,5\},\{3\},\{3,6\},\{1,2\},\{2\},\{2,3\})$ then $w(S) =4533612223$.
Following \cite{Buch2002}, we say that a word $w_1w_2\cdots w_l$ is a \defn{reverse lattice word}
if $\weight(w_iw_{i+1}w_{i+2}\cdots w_l)$ is a partition for each $i \in[l]$.

\begin{lemma}\label{sv-highest-lem}
A set-valued word $S \in \SVWords_n(m)$ has $e'_i(S) = 0$ for all $i \in [n-1]$
if and only if $w(S)$ is a reverse lattice word.
\end{lemma}

\begin{proof}
We have $e'_i(S) \neq 0$ if and only if there is at least one unpaired ``$($'' character in the $i$-word of $S$.
This occurs if and only if there is at least one $i$-unpaired letter (in the sense of Remark~\ref{signa-rem}) equal to $i+1$ in $w(S)$.

Write $w(S) = w_1w_2\cdots w_l$.
For this word to have an $i$-unpaired letter $w_j=i+1$, 
 we must have $\weight(w_{j+1}w_{j+2}\cdots w_l)_i \leq \weight(w_{j+1}w_{j+2}\cdots w_l)_{i+1}$,
but then
$\weight(w_{j}w_{j+1}\cdots w_l)_i < \weight(w_{j}w_{j+1}\cdots w_l)_{i+1}$
so $\weight(w_jw_{j+1}w_{j+2}\cdots w_l)$ is not a partition.
Thus, if $w(S)$ is a reverse lattice word then we must have $e'_i(S) = 0$ for all $i \in [n-1]$.

Conversely, if $w(S)$ is not a reverse lattice word, then for some $i \in [n-1]$,
there must exist $j \in [l]$
with $\weight(w_{j}w_{j+1}\cdots w_l)_i < \weight(w_{j}w_{j+1}\cdots w_l)_{i+1}$.
If $j$ is chosen to be maximal after fixing $i$, then  $w_j = i+1$
(as otherwise the strict inequality would still hold after incrementing $j$ by one),
and this letter is $i$-unpaired in $w(S)$ (since if it were paired with $w_k =i$
for some $k>j$
then the strict inequality would still hold after replacing $j$ by $k+1$).
Thus, if $w(S)$ is not a reverse lattice word, then it has an $i$-unpaired letter equal to $i+1$
for some $i \in [n-1]$, so we have $e'_i(S) \neq 0$.
\end{proof}

The following more straightforward operators are new. Assume $n \geq 2$ and $S \in \SVWords_n(m)$.

\begin{definition}\label{e'1-def}
Construct $e'_{\bar 1}(S) \in \SVWords_n(m)\sqcup \{0\}$  as follows.
If there are no entries in $S$ containing $1$ or $2$, then $e'_{\bar 1}(S) = 0$.
Otherwise, suppose $i \in [m]$ is minimal with $\{1,2\}\cap S_i \neq \varnothing$.
\begin{itemize}
\item  If $1 \notin S_{i}$ and $2\in S_{i}$ then add $1$ to $S_{i}$ to obtain 
$e'_{\bar 1}(S)$. 

\item If $1 \in S_{i}$ and $2 \in S_{i}$ then remove $2$ from $S_{i}$ to obtain 
$e'_{\bar 1}(S)$. 

\item If $1 \in S_{i}$ and $2 \notin S_{i}$ then $e'_{\bar 1}(S) = 0$.

\end{itemize}
\end{definition}

\begin{definition}\label{f'1-def}
Construct $f'_{\bar 1}(S) \in \SVWords_n(m)\sqcup \{0\}$  as follows.
If there are no entries in $S$ containing $1$ or $2$, then $f'_{\bar 1}(S) = 0$.
Otherwise,  suppose $i \in [m]$ is minimal with $\{1,2\}\cap S_i \neq \varnothing$.
\begin{itemize}
\item  If $1 \in S_{i}$ and $2\notin S_{i}$ then add $2$ to $S_{i}$ to obtain 
$f'_{\bar 1}(S)$. 

\item If $1 \in S_{i}$ and $2 \in S_{i}$ then remove $1$ from $S_{i}$ to obtain 
$f'_{\bar 1}(S)$. 

\item If $1 \notin S_{i}$ and $2 \in S_{i}$ then $f'_{\bar 1}(S) = 0$.

\end{itemize}
\end{definition}

\begin{example}
The $f_{\bar 1}'$ operator has this effect on these set-valued words:
\[
\ytabc{.6cm}{
3 \\
13 \\
3 \\
34
}
\ \xrightarrow{f_{\bar 1}'}\ 
\ytabc{.6cm}{
3 \\
123 \\
3 \\
34
}
\ \xrightarrow{f_{\bar 1}'}\ 
\ytabc{.6cm}{
3 \\
23 \\
3 \\
34
} \ \xrightarrow{f_{\bar 1}'}\ 0.
\]
\end{example}

\begin{theorem}\label{sqqn-thm}
For the operators $e'_{i}$ and $f'_{i}$ given above,  $\SVWords_n(m)$ is a $\sqqn$-crystal.
\end{theorem}

The claim that $\SVWords_n(m)$ is a $\sqgln$-crystal essentially follows from \cite[\S4]{Yu};
 the results there technically only consider set-valued words that arise as column reading words of semistandard set-valued tableaux,
 but this property is not used in the relevant proofs. Once 
we know that the set
$\SVWords_n(m)$ is a $\sqgln$-crystal, very little needs to be checked to deduce
  that it is also a $\sqqn$-crystal. We will give a self-contained alternate proof of Theorem~\ref{sqqn-thm} in Section~\ref{sqrt-crystal-tensor-sect}.

\begin{remark} \label{squaring-remk}
Observe that $\SVWords_n(1) \cong \SS_n$ as $\sqqn$-crystals if $\SS_n$ is defined as in Example~\ref{SS-ex}.
Also notice that the $\gl_n$- or $\q_n$-crystal afforded by ``squaring'' $\SVWords_n(m)$ 
in the sense of Proposition~\ref{sq-prop} contains the normal object $\Words_n(m)$ as a union of full subcrystals,
once we identify each ordinary word $w_1w_2\cdots w_m$ with the set-valued word $(\{w_1\},\{w_2\},\dots,\{w_m\})$. Most other full subcrystals of this crystal are not normal as either $\gl_n$-crystals or $\q_n$-crystals.
\end{remark}

\subsection{Square root crystals of tableaux}

Suppose $T$ is a (shifted or unshifted) set-valued tableau
with $m$ boxes and all entries contained in $[n]$.
Define the \defn{set-valued row reading word} of $T$ to be the sequence $\svrow(T) \in \SVWords_n(m)$ formed by listing the set-valued entries of $T$ in the usual row reading order.
Define $\svrevrow(T) \in \SVWords_n(m)$ to be the same sequence read in the opposite order.
Finally, form the \defn{set-valued column reading word}  $\svcol(T) \in \SVWords_n(m)$ by
reading the entries of $T$  down each column in French notation, iterating over columns from left to right.
For example, we have
\[\ba
\svrow\(
\ytabc{.4cm}{
45 \\
3 & 36 \\
12 & 2 & 23
}\) &= (\{4,5\},\{3\},\{3,6\},\{1,2\},\{2\},\{2,3\}),
\\
\svrevrow\(\ytabc{.4cm}{
 \none & \none & 1 \\
 \none & 2 & 1 \\
 34 & 2 & 2 & 23 
}  \) &= (\{2,3\}, \{2\},\{ 2\}, \{3,4\}, \{1\}, \{2\}, \{1\}),
\\
\svcol\(
\ytabc{.4cm}{
45 \\
3 & 36 \\
12 & 2 & 23
}\) &= (\{4,5\},\{3\},\{1,2\},\{3,6\},\{2\},\{2,3\}),
\\
\ea
\]
An \defn{embedding} of $\sqgln$- or $\sqqn$-crystals 
is defined in the same way as for ordinary crystals (see Example~\ref{revrow-thm}):
this means 
a weight-preserving injective map $\phi : \cB \to \cC$  
that commutes with all crystal operators, in the sense that $\phi(e'_i(b)) = e'_i(\phi(b))$ and $\phi(f'_i(b)) = f'_i(\phi(b))$
for all $b \in \cB$ and all relevant indices $i$ when we set $\phi(0) =e'_i(0) = f'_i(0)=0$.

Recall that $ \SetTab_n(\lambda) $ is the set of   semistandard set-valued tableaux of shape $\lambda$
whose entries are all subsets of $[n]$.
The following is equivalent to \cite[Lems.~4.8 and 4.9]{Yu}. 
An example of the resulting crystal
is shown in \cite[Fig.~4.1]{Yu}.

\begin{theorem}[\cite{Yu}]
\label{sqgln-thm}
Let $\lambda$ be a partition.
Then there is a unique $\sqgln$-crystal structure on $\SetTab_n(\lambda)$ for which
$\svcol : \SetTab_n(\lambda) \to \SVWords_n(m)$ is a $\sqgln$-crystal embedding.
\end{theorem}

Our main  new result   is a decomposition tableau analogue of the previous theorem:

\begin{theorem} \label{sqqn-thm2}
Let $\lambda$ be a strict partition.
Then there is a unique $\sqqn$-crystal structure on $\SetDTab_n(\lambda)$ for which
$\svrevrow : \SetDTab_n(\lambda) \to \SVWords_n(m)$ is a $\sqqn$-crystal embedding.
\end{theorem}
 
 The hardest parts of the proof of this theorem will turn out to follow directly from the technical lemmas in Section~\ref{sv-proof-sect}.
 Before explaining this argument, we discuss some examples.

 \begin{example}\label{sqqn-thm-ex}
 Write $e_i'$ and $f_i'$ for the unique maps $\SetDTab_n(\lambda) \to \SetDTab_n(\lambda)\sqcup \{0\}$
 with $\svrevrow\circ e'_i = e'_i \circ \svrevrow$  and $\svrevrow\circ f'_i = f'_i \circ \svrevrow$,
 which exist by Theorem~\ref{sqqn-thm2}.
Then $f_2'$  has this effect on these  set-valued decomposition tableaux:
\[
\ytabc{.4cm}{
 \none & \none & 1 \\
 \none & 2 & 1 \\
 34 & 2 & 2 & 23 
}  
\xrightarrow{f'_2}
\ytabc{.4cm}{
 \none & \none & 1 \\
 \none & 2 & 1 \\
 34 & 23 & 2 & 23 
}  
\xrightarrow{f'_2}
\ytabc{.4cm}{
 \none & \none & 1 \\
 \none & 2 & 1 \\
 34 & 3 & 2 & 23 
}  
\xrightarrow{f'_2}
\ytabc{.4cm}{
 \none & \none & 1 \\
 \none & 2 & 1 \\
 34 & 3 & 23 & 23 
}  
\xrightarrow{f'_2}
\ytabc{.4cm}{
 \none & \none & 1 \\
 \none & 2 & 1 \\
 34 & 3 & 23 & 3 
} \xrightarrow{f'_2}0.
\]
Likewise, the operator $f_{\bar 1}'$ acts as follows:
\[
\ytabc{.6cm}{
 \none & \none & 1 \\
 \none & 2 & 1 \\
 34 & 3 & 13 & 3 
}
\xrightarrow{f_{\bar 1}'}
\ytabc{.6cm}{
 \none & \none & 1 \\
 \none & 2 & 1 \\
 34 & 3 & 123 & 3 
}
\xrightarrow{f_{\bar 1}'}
\ytabc{.6cm}{
 \none & \none & 1 \\
 \none & 2 & 1 \\
 34 & 3 & 23 & 3 
}\xrightarrow{f_{\bar 1}'}0.
\]\end{example}

Figures~\ref{sqrt-q-sv-dec2-fig} and \ref{sqrt-q-sv-dec21-fig} show two $\sqqn$-crystal graphs for $\SetDTab_n(\lambda)$,
while Figures~\ref{q-sv-dec2-fig} and \ref{q-sv-dec21-fig} show the $\q_n$-crystals
 obtained by squaring the crystal operators $e'_i$ and $f'_i$ as in Proposition~\ref{sq-prop}.

 \begin{remark}\label{nsv-rmk}
 In view of Remark~\ref{squaring-remk}, it always holds that the respective subsets of ``non-set-valued'' tableaux in $\SetTab_n(\lambda)$
 and $\SetDTab_n(\lambda) $ (whose entries are all singleton sets)
   are full subcrystals of the $\gl_n$- and $\q_n$-crystals derived from
   Theorems~\ref{sqgln-thm} and \ref{sqqn-thm2} via
    Proposition~\ref{sq-prop}. 
\end{remark}

\begin{figure}
\centerline{\input{sqrt-q-sv-dec2.tex}}
\caption{Crystal graph of the $\sqrt{\q_3}$-crystal $\SetDTab_3(\lambda)$ for $\lambda=(2)$.
Here, 
solid blue and red arrows respectively indicate 
$ \xrightarrow{\ f_1'\ } $ and $ \xrightarrow{\ f_2'\ } $ edges while dashed blue arrows 
indicate $ \xrightarrow{\ f_{\bar 1}'\ }$ edges.
}
\label{sqrt-q-sv-dec2-fig}
\end{figure}

\begin{figure}
\centerline{\input{qsq-sv-dec2.tex}}
\caption{The $\q_3$-crystal 
obtained by squaring the crystal operators in Figure~\ref{sqrt-q-sv-dec2-fig}; compare with Figure~\ref{setdtab-fig1}.
Solid blue and red arrows are
$ \xrightarrow{\ 1\ } $ and $ \xrightarrow{\ 2\ } $ edges. Dashed blue arrows 
are $ \xrightarrow{\ \bar 1\ }$ edges.
Notice that $\DTab_3(\lambda)$ for $\lambda=(2)$ occurs as a connected component of this crystal graph.
}
\label{q-sv-dec2-fig}
\end{figure}

\begin{figure}
\centerline{\input{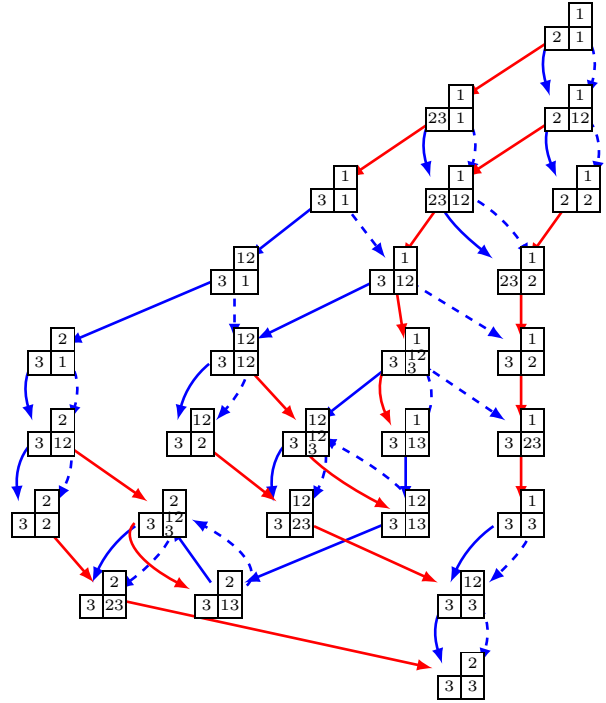}}
\caption{Crystal graph of the $\sqrt{\q_3}$-crystal $\SetDTab_3(\lambda)$ for $\lambda=(2,1)$.
Here, 
solid blue and red arrows respectively indicate 
$ \xrightarrow{\ f_1'\ } $ and $ \xrightarrow{\ f_2'\ } $ edges while dashed blue arrows 
indicate $ \xrightarrow{\ f_{\bar 1}'\ }$ edges.
}
\label{sqrt-q-sv-dec21-fig}
\end{figure}

\begin{figure}
\centerline{\input{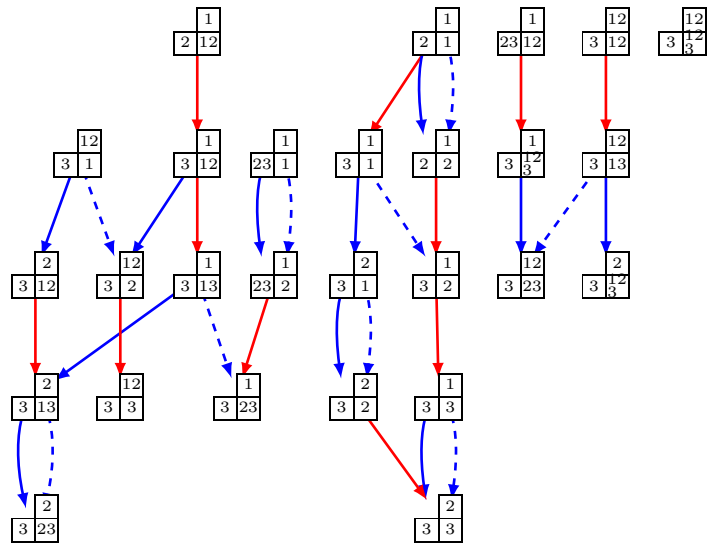}}
\caption{The $\q_3$-crystal obtained by squaring the crystal operators in Figure~\ref{sqrt-q-sv-dec21-fig}; compare with Figure~\ref{setdtab-fig2}.
Solid blue and red arrows are
$ \xrightarrow{\ 1\ } $ and $ \xrightarrow{\ 2\ } $ edges. Dashed blue arrows 
are $ \xrightarrow{\ \bar 1\ }$ edges.
Notice that $\DTab_3(\lambda)$ for $\lambda=(2,1)$ occurs as a connected component of this crystal graph.
}
\label{q-sv-dec21-fig}
\end{figure}

\begin{proof}[Proof of Theorem~\ref{sqqn-thm2}]
Fix $T \in \SetDTab_n(\lambda)$. When $e'_i(\svrevrow(T))$ 
 is nonzero, there is certainly a unique set-valued tableau $U$
of shape $\lambda$ with $\svrevrow(U) = e'_i(\svrevrow(T))$,
although it is not obvious that this is a decomposition tableau.
Set $e'_i(T) := U$ when $e'_i(\svrevrow(T))\neq 0$ and define $e'_i(T) := 0$ otherwise.
Form $f'_i(T)$ analogously. 
The most difficult thing to check is that these
operators $e'_{i}$ and $f'_{i}$ really do give maps $\SetDTab_n(\lambda) \to \SetDTab_n(\lambda)\sqcup\{0\}$.

For this,
first suppose $i \in [n-1]$. Observe that  $e'_i$ or $f'_i$ act on $T$ by either removing 
a letter from some box or adding a letter to some box, when they do not send $T \mapsto 0$.
In the cases when $e'_i$ or $f'_i$ remove a letter from some box,
it is clear $e'_i(T) $ and $f'_i(T)$ belong to $\SetDTab_n(\lambda)$.
On the other hand, if the $i$-word of $\svrevrow(T)$ has no combined forms
and at least one left form, then the number that corresponds to the ``('' at the start of the first left form is also the first $i$-unpaired letter equal to $i+1$ in $\revrow(T)$ in the sense of Lemma~\ref{e-op-SetDTab}.
Changing this $i+1$ to $i$ must yield an element of $ \SetDTab_n(\lambda)$
since otherwise parts (2) and (3) of Lemma~\ref{e-op-SetDTab} would imply that the ``(''  character being considered is not the start of a left form  in the linked $i$-word. Therefore  adding $i$ to the box of $i+1$ also gives an element of 
$ \SetDTab_n(\lambda)$.

Similarly, if the   $i$-word of $\svrevrow(T)$ has no combined forms
and at least one right form, then the number that corresponds to the ``)'' at the start of the last right form is also the last $i$-unpaired letter equal to $i$ in $\revrow(T)$ in the sense of Lemma~\ref{f-op-SetDTab}.
Changing this $i$ to $i+1$ must yield an element of $ \SetDTab_n(\lambda)$
since otherwise parts (2) and (3) of Lemma~\ref{f-op-SetDTab} would imply that the ``)''  character being considered is not the start of a right form  in the   $i$-word. Therefore  adding $i+1$ to the box of $i$ also gives an element of 
$ \SetDTab_n(\lambda)$.

One can either deduce from Theorem~\ref{sqqn-thm}
that $e'_i(T) = U\neq 0$ if and only if $T=f'_i(U)$,
or derive this fact 
repeating the argument in the proof of \cite[Lem.~4.9]{Yu}.
Since $\varepsilon_i'(T)$ is twice the number of left forms in the $i$-word of $\svrevrow(T)$ plus one if there is a combined form,
and $\varphi'_i(T)$ is twice the number of right forms in the $i$-word of $\svrevrow(T)$ plus one if there is a combined form,  the sum $\varepsilon_i'(T) + \varphi_i'(T)$ is even. Moreover,
$\frac{\varphi_i'(T) - \varepsilon_i'(T)}{2}$ is exactly the difference between the number of unpaired letters equal to $i$ and $i+1$ in the   $i$-word of $\svrevrow(T)$, which is just $\weight(T)_i - \weight(T)_{i+1}$.
Finally, the weight identity in part (b) of Definition~\ref{sqgln-def} is immediate.

This shows that $\SetDTab_n(\lambda)$ is a $\sqgln$-crystal.
We move on to the conditions in Definition~\ref{sqqn-def}, assuming $n\geq 2$.
If $e'_{\bar 1}$ or $f'_{\bar 1}$ acts on $T$ by removing a 2 or 1 from some box
then clearly $e'_{\bar 1}(T) $ and $f'_{\bar 1}(T) $ still belong to $\SetDTab(\lambda)$.
Suppose $(i,j)$ is the first box of $T$  in the reverse row reading word 
containing $1$ or $2$. If we have $1\notin T_{ij}$ and $2\in T_{ij}$,
then $e_{\bar 1}$ acts on every distribution of $T$ containing  $1\in T_{ij}$
by changing this 1 to 2; therefore adding $2$ to $T_{ij}$ gives an element of $\SetDTab_n(\lambda)$.
Likewise,  if we have $1 \in T_{ij}$ and $2\notin T_{ij}$, 
then $f_{\bar 1}$ acts on every distribution of $T$ containing  $2\in T_{ij}$
by changing this 2 to 1; therefore adding $1$ to $T_{ij}$ again gives an element of $\SetDTab_n(\lambda)$.
The remaining things to check in Definition~\ref{sqqn-def}
are straightforward. 
 \end{proof}

It follows from \cite[Cor.~5.2]{Yu} that if $\lambda$ is a partition with at most $n$ parts then
the $\sqgln$-crystal $\SetTab_n(\lambda)$ is connected (in the sense that its crystal graph
has only one weakly connected component). More strongly, Yu's result show that $\SetTab_n(\lambda)$ has
 a unique highest weight element (namely, the tableau of shape $\lambda$ with the set
 $\{i\}$ in every entry in row $i$)
 that is sent to zero by all $e'_i$ operators.
We predict that a similar phenomenon holds for the $\sqqn$-crystal $\SetDTab_n(\lambda)$:

\begin{conjecture}\label{sqqn-conj-0}
For each strict partition $\lambda$,
the $\sqqn$-crystal $\SetDTab_n(\lambda)$ is connected.
\end{conjecture}

This is supported by all examples that we can compute.  
These examples also exhibit a more technical highest weight property, which we now explain.

Let $\lambda$ be a strict partition with $\ell(\lambda)=k$ nonzero parts.
The \defn{first border strip} of the shifted diagram $\SD_\lambda$
is the minimal set of positions $S$ such that
\bei
\item[(a)] one has $(1,\lambda_1) \in S$ and  
\item[(b)] if $(i,j) \in S$ and $i\neq j$, then either
$(i+1,j)\in S$, or 
$(i,j-1) \in S$ when $(i+1,j)\notin \SD_\lambda$. 
\eei
Let $\SD_\lambda^{(1)}$ be the first border strip of $\SD_\lambda$.
The set difference $\SD_\lambda- \SD_\lambda^{(1)}$ is either empty when $k=1$
or equal to $\SD_\mu$ for a strict partition $\mu$ with $\ell(\mu) = k-1$.
For $i\in[k-1]$ let $\SD_\lambda^{(i+1)}$ be the first border strip of $\SD_\lambda - (\SD_\lambda^{(1)} \sqcup \cdots \sqcup \SD_\lambda^{(i)})$.
Finally, let $\Thighest_\lambda\in \SetDTab_n(\lambda)$ be the set-valued decomposition tableau of shape $\lambda$ whose entries in each border strip $\SD_\lambda^{(i)}$ are all $\{i\}$.

\begin{example}If $\lambda = (6,4,2,1)$ then
$\Thighest_{\lambda} = \ytab{
\none & \none & \none & 1 \\
\none & \none & 2 & 1 \\ 
\none & 3 & 2 & 1 & 1 \\
4 & 3 & 2 & 2 & 1 & 1
}
$.
\end{example}

Assume $\cB$ is a $\sqqn$-crystal.
An \defn{$i$-string} in $\cB$ is a connected component in the subgraph of its crystal graph retaining only the $\xrightarrow{i}$ arrows. Let $\sigma'_i : \cB \to \cB$ be the involution that reverses each $i$-string. Then define
$e'_{\bar i}: \cB \to \cB \sqcup\{0\}$  
for indices $2\leq i <n$ by
\be
\label{sq-bar-i-eq} 
\ba
e'_{\bar i} &:=  (\sigma'_{i-1}   \sigma'_i)    \cdots  (\sigma'_2 \sigma'_3) (\sigma'_1 \sigma'_2) e'_{\bar 1} (\sigma'_2 \sigma' _1) (\sigma'_3 \sigma'_2) \cdots (\sigma'_i   \sigma'_{i-1}),
\ea
\ee
using the convention that $\sigma'_i(0) = 0$. 
Finally, let us say that an element $b\in\cB$ is \defn{$\sqqn$-highest weight}
if $e'_i (b) =e'_{\bar i}(b)= 0$ for all $i \in [n-1]$.

This construction is motivated by the definition of a \defn{highest weight} element in a $\q_n$-crystal \cite[Def.~1.12]{GJKKK}, which is specified in the same way via the ordinary $\q_n$-crystal operators.
It is known from \cite[Thm.~2.5]{GJKKK} that
 $\Thighest_\lambda$ is the unique highest weight element of the $\q_n$-crystal 
$\DTab_n(\lambda)$ from Example~\ref{revrow-thm}.
This implies the following non-obvious fact:

\begin{proposition}
The tableau  $\Thighest_\lambda$ is a $\sqqn$-highest weight element of $\SetDTab_n(\lambda)$.
\end{proposition}

\begin{proof}
Let $\sigma_i$ be the operator that reverses all $i$-strings in the $\q_n$-crystal $\DTab_n(\lambda)$,
and identify $\DTab_n(\lambda)$ with a subset of $\SetDTab_n(\lambda)$.
Recall from Remark~\ref{nsv-rmk} that the crystal operators $e_i$ and $f_i$ on $\DTab_n(\lambda)$ 
are related to  $e'_i$ and $f'_i$ by the relations
 $e_i = e'^2_i$ and $f_i = f'^2_i$.

Fix $T \in \DTab_n(\lambda)$ and $i \in [n-1]$.
Then each $i$-word of $\svrevrow(T)$ has no combined forms, so if $e'_i(T)$   is nonzero then the $i$-word of
$\svrevrow(e'_i(T))$ does have a combined form, as does the $i$-word of
$\svrevrow(f'_i(T))$ if $f'_i(T)\neq 0$.
It follows that for either operator $\phi \in \{e'_i, f'_i\}$, we have  $\phi(T) \neq 0$ if and only if $\phi^2(T)\neq 0$, in which case $\phi^2(T)$ is also in $\DTab_n(\lambda)$.
This implies that the $i$-string through $T$ in $\SetDTab_n(\lambda)$ 
contains an odd number of vertices, which alternate between elements of 
 $\DTab_n(\lambda)$ and $\SetDTab_n(\lambda)\setminus \DTab_n(\lambda)$, and   we have
$\sigma_i(T) =\sigma'_i(T) \in \DTab_n(\lambda)$.

It is also true for $T \in \DTab_n(\lambda)$ that $e'_{\bar 1}(T) \neq  0$ if and only if $e_{\bar1}(T) = e'^2_{\bar 1}(T) \neq 0$, as both situations occur precisely when the reverse row reading word of $T$ contains a $2$ before any $1$.
We conclude that $T \in \DTab_n(\lambda)$ satisfies $e'_i (T) =e'_{\bar i}(T)= 0$ for all $i \in [n-1]$
if and only if $e_i (T) =e_{\bar i}(T)= 0$ for all $i \in [n-1]$, where $e_{\bar i} :=  (\sigma_{i-1}   \sigma_i)    \cdots  (\sigma_2 \sigma_3) (\sigma_1 \sigma_2) e_{\bar 1} (\sigma_2 \sigma _1) (\sigma_3 \sigma_2) \cdots (\sigma_i   \sigma_{i-1})$.
As already mentioned, results in \cite{GJKKK} imply that $\Thighest_\lambda$ satisfies the second set of conditions.
\end{proof}

The following stronger version of the previous result holds in all of our examples.
 Note that as $\sqqn$-crystal graphs may have cycles, this property does not immediately imply Conjecture~\ref{sqqn-conj-0}.
 
\begin{conjecture}\label{sqqn-conj}
The tableau  $\Thighest_\lambda$ is the unique $\sqqn$-highest weight element of $\SetDTab_n(\lambda)$.
\end{conjecture}

We can verify both conjectures in the following special cases:

\begin{proposition}
Conjectures~\ref{sqqn-conj-0} and \ref{sqqn-conj} hold if
$n\leq 2$ or $\ell(\lambda) \leq 1$.
\end{proposition}

Beyond this, we have checked the two conjectures
by computer when either
$n=3$ and $|\lambda| \leq 20$,
$n=4$ and $|\lambda| \leq 10$,
$n=5$ and $|\lambda| \leq 6$,
$n=6$ and $|\lambda| \leq 5$,
or $n\in\{7,8\}$ and $|\lambda| \leq 3$.

\begin{proof}
We will prove the stronger statement that if $n\leq 2$ or $\ell(\lambda) \leq 1$,
then for each $T \in  \SetDTab_n(\lambda)$ there are indices $i_1,i_2,\dots,i_q \in \{\bar 1,1,2,\dots,n-1\}$
with $e'_{i_1} e'_{i_2} \cdots e'_{i_q}(T) = \Thighest_\lambda$.

If $n=1$ or $\ell(\lambda)=0$ then  $\Thighest_\lambda$ is the unique element
of $ \SetDTab_n(\lambda)$,
so the claim is immediate.
Choose any $T \in \SetDTab_n(\lambda)$. Next assume $n=2$.
In view of Lemma~\ref{decomptab-characterisation}, 
 $T$ has at most two rows and all boxes in its second row must contain the set $\{1\}$.
Suppose there are $k$ boxes in the second row.
Then the first row $T$ read left to right must consist of 
$k$ or more boxes containing $\{2\}$, 
optionally followed by a box containing $\{1,2\}$,
then zero or more boxes containing $\{1\}$, 
optionally followed by a final box containing $\{2\}$ or $\{1,2\}$.

We now explain how to apply a sequence to raising operators to turn $T$ into $\Thighest_\lambda$.
If the last box in the first row of $T$ contains  $\{1,2\}$ (respectively, $\{2\}$)
then we apply $e'_{\bar 1}$ once (respectively, twice) to change this box's entry to $\{1\}$.
Next, if there is a box in the middle of the first row containing $\{1,2\}$,
then the $1$-word of $\svrevrow(T)$ will have a combined form and we can apply $e'_1$
to change the middle box's entry to $\{1\}$.
After applying these operators, the first row of $T$
consists of some number $l $ boxes containing $\{2\}$, with $l\geq k$, followed by zero or more boxes containing $\{1\}$. Applying $(e'_1)^{2(l-k)}$ turns $T$ into $\Thighest_\lambda$ as needed.

Finally suppose $n> 2$ but $\ell(\lambda)=1$, so that $\lambda=(m)$ for 
 some $m \in \PP$. We prove our claim by induction on the set of entries $\cE =\cE(T) := \bigcup_{(i,j) \in T} T_{ij}$ appearing in $T$,
 using graded lexicographic order.
If $\cE = \{1\}$, then $T =\Thighest_\lambda$.
Suppose $\cE \neq \{1\}$ but $i := \min(\cE \setminus\{1\})$ is greater than $2$.
If the number $i$ appears exactly $k$ times in $T$, then 
applying
 $(e'_{i-1})^{2k}$ to $T$ will change all of these $i$ entries to $i-1$, and the set $\cE$ will become $(\cE \setminus\{i\}) \sqcup \{i-1\}$. Then by induction we can find a sequence of raising operators that will turn $(e'_{i-1})^{2k}(T)$ into $\Thighest_\lambda$.
 
Assume that $2$ appears in some entry of $T$. 
Since $T$ is a set-valued decomposition tableaux, the entries 
of $T$ containing $1$ or $2$
 must occur in a contiguous sequence of boxes,
  say in columns $a+1,a+2,\dots,b$ where $0\leq a<b$.
  If we consider just these boxes and remove all numbers greater than 2,
then  we obtain a set-valued decomposition tableau $U\in \SetDTab_2((b-a))$
     that is not equal to $\Thighest_{(b-a)}$. The argument given in the $n=2$ case above
shows that there is a sequence of indices $i_1,i_2,\dots,i_q \in \{\bar 1,1\}$ with $e'_{i_1} e'_{i_2} \cdots e'_{i_q}(U) = \Thighest_{(b-a)}$. Clearly for these indices we also have $\cE(e'_{i_1} e'_{i_2} \cdots e'_{i_q}(T)) = (\cE \setminus\{2\})\cup \{1\}$, and so by induction we can find additional raising operators that will turn $e'_{i_1} e'_{i_2} \cdots e'_{i_q}(T)$ into $\Thighest_\lambda$ as needed.
\end{proof}

Yu \cite{Yu} also demonstrates a close relationship between his $\sqgln$-crystals $\SetTab_n(\lambda)$
and \defn{Lascoux polynomials}. It would be interesting to know if there is a similar relationship between
the $\sqqn$-crystals $\SetDTab_n(\lambda)$ and the \defn{$P$-Lascoux polynomials} introduced in \cite{MarScr}.

\subsection{Tensor products}\label{sqrt-crystal-tensor-sect}

This section investigates a tensor product for 
 $\sqgln$- and $\sqqn$-crystals.
 This can be defined by reusing 
the same tensor product rules  as for $\gl_n$- and $\q_n$-crystals.
The resulting construction is algebraically very natural, although it 
is no longer (directly) motivated by the representation theory of an associated quantum group.

\begin{theorem}\label{sq-tensor-thm}
Suppose $\cB$ and $\cC$ are $\sqgln$-crystals.
Then the set $\cB \otimes \cC = \{ b \otimes c : b\in \cB,\ c\in \cC\}$
 has a unique $\sqgln$-crystal structure
in which
$
\weight(b\otimes c) = \weight(b) + \weight(c)
$
and  
\[\ba 
e'_i(b\otimes c) &= \begin{cases}
b \otimes e'_i(c) &\text{if }\varepsilon'_i(b) \leq \varphi'_i(c), \\
e'_i(b) \otimes c &\text{if }\varepsilon'_i(b) > \varphi'_i(c),
\end{cases}
\\
f'_i(b\otimes c) &= \begin{cases}
b \otimes f'_i(c) &\text{if }\varepsilon'_i(b) < \varphi'_i(c), \\
f'_i(b) \otimes c &\text{if }\varepsilon'_i(b) \geq \varphi'_i(c).
\end{cases}
\ea
\]
If $\cB$ and $\cC$ are $\sqqn$-crystals,
then the $\sqgln$-crystal $\cB\otimes \cC$ has a unique $\sqqn$-crystal structure 
with
\[\ba
  e'_{\overline 1}(b\otimes c) &= \begin{cases} 
 b \otimes e'_{\overline 1}(c)&\text{if }e'_{\overline1}(b) = f'_{\overline1}(b) = 0,
 \\
  e'_{\overline 1}(b) \otimes c
&\text{otherwise},
 \end{cases}
\\
  f'_{\overline 1}(b\otimes c) &= \begin{cases} 
 b \otimes f'_{\overline 1}(c)&\text{if }e'_{\overline1}(b) = f'_{\overline1}(b) = 0,
 \\
  f'_{\overline 1}(b) \otimes c
&\text{otherwise}.
 \end{cases}
 \ea\]
Finally,
  the natural maps $\cB \otimes (\cC \otimes \cD) \to (\cB \otimes \cC) \otimes \cD$ are
  $\sqgln$- and $\sqqn$-crystal isomorphisms.
\end{theorem}

\begin{proof}
Fix $i\in[n-1]$, $b \in \cB$, and $c \in \cC$. 
Then 
\[\varepsilon'_i(b\otimes c) = \max\{0, \varepsilon'_i(b) - \varphi'_i(c)\} + \varepsilon'_i(c)
\quand
\varphi'_i(b\otimes c) = \max\{0,  \varphi'_i(c) - \varepsilon'_i(b)\} + \varphi'_i(b).\]
So we have
$\varepsilon'_i(b\otimes c)  + \varphi'_i(b\otimes c)
\equiv 
\varepsilon'_i(b )  + \varphi'_i(b ) + \varepsilon'_i( c)  + \varphi'_i( c)\equiv 0 \modu 2)$
as well as
\[
\tfrac{\varphi'_i(b\otimes c)  - \varepsilon'_i(b\otimes c)}{2}=
\tfrac{\varphi'_i(b )  - \varepsilon'_i(b )}{2} + \tfrac{\varphi'_i( c)  - \varepsilon'_i( c)}{2}
= \weight(b\otimes c)_i - \weight(b\otimes c)_{i+1}
.\]

It is clear that $e'_i$ and $f'_i$ act as inverse operators 
when they are nonzero.
Suppose   $e'_i(b\otimes c)\neq 0$. If $\varepsilon_i'(b) \leq \varphi'_i(c)$ then 
$
\weight(e'_i(b\otimes c)) - \weight(b\otimes c)
=
\weight(e'_i(c)) - \weight(c)
$
and $ \varepsilon_i'(b\otimes c) = \varepsilon_i'(c).$
If $\varepsilon_i'(b) > \varphi'_i(c)$ then 
$
\weight(e'_i(b\otimes c)) - \weight(b\otimes c)
=
\weight(e'_i(b)) - \weight(b)
$ and \[ \varepsilon_i'(b\otimes c) =
\varepsilon'_i(b) - \varphi'_i(c) + \varepsilon'_i(c)
\equiv \varepsilon'_i(b) +\varphi'_i(c) + \varepsilon'_i(c)
\equiv 
 \varepsilon_i'(b)\modu 2).\]
Therefore, property (b) in Definition~\ref{sqgln-def} for $\cB$ and $\cC$
implies the same property for $\cB \otimes \cC$. We conclude that 
$\cB \otimes \cC$ is a $\sqgln$-crystal.

The conditions we need to check to show that $\cB\otimes \cC$ 
is a $\sqqn$-crystal when $\cB$ and $\cC$ are $\sqqn$-crystals 
are all straightforward. Notice in particular that the first two components of the weight of $b\otimes c$ are both zero
if and only if the same is true of both $b$ and $c$, as we require all weights to be in $ \NN^n$.

Finally, the argument to show that the natural map $\cB \otimes (\cC \otimes \cD) \to (\cB \otimes \cC) \otimes \cD$ is a
  $\sqgln$- or $\sqqn$-crystal isomorphism is the same as the proof of
  the associativity of the tensor product for ordinary $\gl_n$- or $\q_n$-crystals; see \cite[\S2.3]{BumpSchilling}.
\end{proof}

\begin{remark}
The tensor products for $\gl_n$- and $\sqgln$-crystals
(respectively, $\q_n$- and $\sqqn$-crystals)
do not commute with the ``squaring'' operation in Proposition~\ref{sq-prop}:
if $\cB^{(2)}$ denotes the 
crystal obtained from $\cB$ via that result,
then one can have $\cB^{(2)} \otimes \cC^{(2)} \not\cong (\cB\otimes \cC)^{(2)}$.
It does hold as usual that   $\ch(\cB\otimes \cC) = \ch(\cB)\ch(\cC)$ when $\cB$ and $\cC$ are both finite, however.
\end{remark}

Fix $m \in \NN$ and recall the definition of $\SS_n$ from Example~\ref{SS-ex}. There is an obvious weight-preserving bijection
$\SS_n^{\otimes m} \to \SVWords_n(m)$ given by identifying tensors with tuples, which sends \be\label{obvSS} 
S_1\otimes S_2 \otimes \cdots \otimes S_m \mapsto (S_1,S_2,\dots,S_m) \in \SVWords_n(m).
\ee
The following theorem shows that the $\sqqn$-crystal structure on $\SVWords_n(m)$,
which may have seemed unmotivated in Section~\ref{sq-words-sect}, is in fact the   one induced by this bijection. This statement implies Theorem~\ref{sqqn-thm}, and so provides an
alternate proof that result.

\begin{theorem}\label{sq-tensor-thm2}
The bijection \eqref{obvSS} is an isomorphism of $\sqqn$-crystals $\SS_n^{\otimes m} \xrightarrow{\sim} \SVWords_n(m)$.
\end{theorem}

\begin{proof}
Fix $m\geq 2$ and $T= (T_1,T_2,\dots,T_m) \in \SVWords_n(m)$. Let $U = (T_2,\dots,T_m)$.
Recall that $\SS_n \cong \SVWords_n(1)$ as $\sqqn$-crystals.
To show that the map $\SS_n^{\otimes m} \to \SVWords_n(m)$ is an isomorphism of $\sqgln$-crystals, it suffices by induction on $m$ to check 
for each $i \in [n-1]$ that 
\[
e'_i(T) = \begin{cases}
(T_1, e'_i(U)) &\text{if }\varepsilon'_i(T_1) \leq \varphi'_i(U) \\
(e'_i(T_1),U) &\text{if }\varepsilon'_i(T_1) > \varphi'_i(U)
\end{cases}
\quand
f'_i(T) = \begin{cases}
(T_1, f'_i(U)) &\text{if }\varepsilon'_i(T_1) < \varphi'_i(U) \\
(f'_i(T_1) , U) &\text{if }\varepsilon'_i(T_1) \geq \varphi'_i(U),
\end{cases}
\]
where we set 
\[(T_1, e'_i(U)):=\begin{cases} 0 &\text{if $e'_i(U)=0$} \\
(T_1,V_2,\dots,V_m)&\text{if $e'_i(U) = (V_2,\dots,V_m)\neq 0$,}
\end{cases}\]
 and interpret  $(e'_i(T_1),U)$, $(T_1, f'_i(U))$,
and $(f'_i(T_1) , U)$ similarly. We will just prove the formula for $e'_i(T)$ since the argument 
for $f'_i(T)$ is not much different.

Recall that $\varphi'_i(U)$ is twice the number of right forms in the $i$-word of $U$ plus one if there is a combined form.
On the other hand, $\varepsilon'_i(T_1)$ is $2$ if $T_1 \cap \{i,i+1\} =\{i+1\}$,
$1$ if $T_1 \cap \{i,i+1\} = \{i,i+1\}$, and $0$ otherwise.
The three possibilities here combine with the two in our desired formula for $e'_i(T)$ to give six different cases:
\bei
\item[(1)] Suppose $\varepsilon'_i(T_1)=2 \leq \varphi'_i(U)$.
Then the $i$-word of $U$ must have at least one right form,
and the first such form will merge with the ``$($'' contributed by $T_1$ to create a null form in the $i$-word of $T$. In this case,
the $i$-word of $T$ will have a combined form (respectively, a left form) if and only if the $i$-word of $U$ does as well, and the entries that correspond to ``$)-($'' at the end of the combined forms (respectively, ``$($'' at the start of the first left forms) in $T$ and $U$ coincide,
so $e'_i(T) = (T_1, e'_i(U))$. 

\item[(2)] Suppose $\varepsilon'_i(T_1)=1 \leq \varphi'_i(U)$.
Then the $i$-word of $U$ must have a right  or combined form.
If there is a right form, then it will   merge with the ``$)-($'' contributed by $T_1$ to create the first right form in the $i$-word of $T$,
and
 it follows exactly as in case (1) that $e'_i(T) = (T_1, e'_i(U))$.
If there are no right forms, then the unique   combined form in the $i$-word of $U$ 
 will merge with the ``$)-($'' contributed by $T_1$ to create a larger combined form in the $i$-word of $T$. When this happens, the entries that correspond to ``$)-($'' at the end of the combined forms  in $T$ and $U$ coincide,
so again $e'_i(T) = (T_1, e'_i(U))$.

\item[(3)] Suppose $\varepsilon'_i(T_1)=0 \leq \varphi'_i(U)$.
Then the $i$-word of $T$ is either identical to the $i$-word of $U$,
or is given by appending the $i$-word of $U$ after the single right form  ``$)$''.
Either way, the same reasoning as in case (1) shows that  $e'_i(T) = (T_1, e'_i(U))$.

 \item[(4)] Suppose $\varepsilon'_i(T_1)=2 > \varphi'_i(U)$.
 Then the $i$-word of $U$ must have no right forms.
 If the $i$-word of $U$ has a combined form, then it will merge with the ``$($'' contributed by $T_1$
 to create the first left form in the $i$-word of $T$, which will have no combined forms.
 If there is no combined form in the $i$-word of $U$ then the ``$($'' contributed by $T_1$
 will become on its own the first left in the $i$-word of $T$.
 In both situations, the $i$-word of $T$ has no combined forms and $T_1$ is the entry that corresponds to the ``$($'' at the start of the first left form, so   $e'_i(T) = (e'_i(T_1), U)$.
 
  \item[(5)] Suppose $\varepsilon'_i(T_1)=1 > \varphi'_i(U)$.
   Then the $i$-word of $U$ must have no right or combined forms,
   so  the ``$($'' contributed by $T_1$ will become  the first left in the $i$-word of $T$,
   which will also have no combined forms.
   Then, it follows exactly as in case (4) that $e'_i(T) = (e'_i(T_1), U)$.
   
     \item[(6)] The final case $\varepsilon'_i(T_1)=0 > \varphi'_i(U)$ is impossible since $\varphi'(U) \in \NN$.
   
\eei
This cases analysis shows that $e'_i(T)$ does have the desired formula in the equation displayed above.
We conclude that $\SS_n^{\otimes m} \to \SVWords_n(m)$ is an isomorphism of $\sqgln$-crystals.

To upgrade this to a $\sqqn$-isomorphism,
it suffices (again by induction on $m$)  just   to show that
\[
  e'_{\overline 1}(T) = \begin{cases} 
 (T_1,   e'_{\overline 1}(U))&\text{if }1,2\notin T_1 
 \\
(  e'_{\overline 1}(T_1) ,U)
&\text{otherwise}
 \end{cases}
\quand
  f'_{\overline 1}(T) = \begin{cases} 
 (T_1,   f'_{\overline 1}(U))&\text{if }1,2\notin T_1
 \\
 ( f'_{\overline 1}(T_1) , U)
&\text{otherwise}.
 \end{cases}
 \]
These formulas are more or less immediate from the relevant
Definitions~\ref{e'1-def} and \ref{f'1-def}.
\end{proof}

\begin{corollary}\label{pq-cor}
If $p,q \in \NN$ then $\SVWords_n(p)\otimes \SVWords_n(q) \cong \SVWords_n(p+q)$.
\end{corollary}

We mention an interesting $\sqgln$-crystal theoretic interpretation
of the coefficients  in the symmetric Grothendieck expansion of the product $G_\lambda G_\mu = \sum_\nu c_{\lambda\mu}^\nu G_\nu$.
 The coefficients $ c_{\lambda\mu}^\nu$ in this expansion were first shown to be nonnegative integers by Buch \cite{Buch2002}.
 Using Lemma~\ref{sv-highest-lem} and Corollary~\ref{pq-cor}, we can reformulate 
 Buch's description of these numbers as follows:

\begin{theorem}[Buch \cite{Buch2002}]
\label{buch-cor}
 If $\lambda$ and $\mu$ are partitions then
\[G_\lambda(x_1,x_2,\dots,x_n)G_\mu(x_1,x_2,\dots,x_n) =
\sum_{
\substack{b \in \SetTab_n(\lambda)\otimes \SetTab_n(\mu) \\
e'_i(b) =0\ \forall i \in[n-1]}}
G_{\weight(b)}(x_1,x_2,\dots,x_n).
\]
\end{theorem}

\begin{proof}
We attribute this theorem to Buch as it is almost  immediate from \cite[Thm.~5.4]{Buch2002}.
That result
is equivalent to the statement that 
$
\textstyle G_\lambda(x_1,\dots,x_n) G_\mu(x_1,\dots,x_n) = \sum G_{\weight(T\otimes U)} (x_1,\dots,x_n)$
where the sum is over 
 tensors $T \otimes U \in \SetTab_n(\lambda)\otimes \SetTab_n(\mu)$  such that $w(\svcol(T)\svcol(U))$ is a reverse lattice word, with $w(T)$ defined as in  Lemma~\ref{sv-highest-lem}.
By Corollary~\ref{pq-cor}, we can embed \[\SetTab_n(\lambda)\otimes \SetTab_n(\mu) \subseteq 
\SVWords_n(|\lambda|)\otimes \SVWords_n(|\mu|) \cong \SVWords_n(|\lambda|+|\mu|) \]
and then identify $T\otimes U$ with the concatenation $\svcol(T)\svcol(U)$,
so it follows from Lemma~\ref{sv-highest-lem} that $w(\svcol(T)\svcol(U))$ is a reverse lattice word 
if and only if $e'_i(T\otimes U) =0$ for all $i \in[n-1]$.
\end{proof}

The preceding result does not lift to crystals, as
the full subcrystal of   
$T\in \SetTab_n(\lambda)\otimes \SetTab_n(\mu)$  with $e'_i(T)=0$ for all $ i \in[n-1]$
is not always isomorphic to $\SetTab_n(\weight(T))$.

\begin{remark}\label{buch-rmk}
If $\lambda$ and $\mu$ are strict partitions then
$\GP_\lambda(x_1,\dots,x_n)\GP_\mu(x_1,\dots,x_n)$ is an
$\NN$-linear combination $\GP$-functions $\GP_\nu(x_1,\dots,x_n)$ \cite{CTY,PechenikYong}.
However, the coefficients in this expansion appear to be overcounted
by the $\sqqn$-highest weights  in 
$\SetDTab_n(\lambda)\otimes \SetDTab_n(\mu)$, at least as defined in Conjecture~\ref{sqqn-conj}. To be precise, in all examples   we can compute, it holds that
\be\label{GPGP}
\GP_\lambda(x_1,\dots,x_n)\GP_\mu(x_1,\dots,x_n) \preceq
\sum_{
\substack{b \in \SetDTab_n(\lambda)\otimes \SetDTab_n(\mu) \\
e'_i(b) =e'_{\overline{i}}(b)=0\ \forall i \in[n-1]}}
\GP_{\weight(b)}(x_1,\dots,x_n)
\ee
where we write $f \preceq g$ to
mean $g-f$ is an $\NN$-linear combination of $\GP$-polynomials. 
\end{remark}

We use the tensor products in Theorem~\ref{sq-tensor-thm} to introduce families of \defn{normal $\sqgln$-crystals} and  \defn{normal $\sqqn$-crystals}:
 we define these to be the smallest full subcategories 
of crystals 
  that contain all tensor powers of
the standard object $\SS_n$ and that are closed under disjoint unions and restriction to any connected component of the crystal graph.
Theorems~\ref{sqgln-thm}, \ref{sqqn-thm2}, and \ref{sq-tensor-thm2}
imply that:
 
 \begin{corollary}
 The objects $\SVWords_n(m)$ and $\SetDTab_n(\lambda)$ are   normal as both $\sqgln$- and $\sqqn$-crystals,
 while
$\SetTab_n(\lambda)$ is a normal $\sqgln$-crystal.
 \end{corollary}
 
From Lemma~\ref{sv-highest-lem}, we also know that:

\begin{corollary}
In a normal $\sqgln$-crystal,
the weight of any highest weight element (that is, an element $b$ with $e'_1(b) =e'_2(b) =\dots= e'_{n-1}(b)=0$) is a partition with at most $n$ nonzero parts.
\end{corollary}

The categories of normal $\sqgln$- and $\sqqn$-crystals are automatically closed under tensor products and disjoint unions. However, they are not so well-behaved
as their classical counterparts.
Every connected normal 
$\gl_n$- or $\q_n$-crystal has a unique highest weight element,
whose weight uniquely determines the crystal's isomorphism class,
and all finite normal $\gl_n$- or $\q_n$-crystals with the same character are isomorphic (see \cite[\S1.2]{MT2023}).
By contrast:
\begin{itemize}
\item There are
connected normal $\sqgln$- and $\sqqn$-crystals with multiple highest weight elements.
For example, there is a full $\sqrt{\gl_3}$-subcrystal of $\SVWords_3(4)$ with two highest weight elements
\[ (\{1,3\}, \{1\},\{1,2\},\{1\})  \quand
(\{1,2,3\}, \{1\},\{1,2\},\{1\}) 
\]
whose character is $G_{(4,1,1)}(x_1,x_2,x_3)+G_{(4,2,1)}(x_1,x_2,x_3)$,
and there is a 
 full $\sqqn$-subcrystal of $\SVWords_3(5)$ with two highest weight elements 
\[ (\{1\}, \{2\},\{1\},\{1,2\},\{1\})  \quand
(\{1\}, \{2,3\},\{1\},\{1,2\},\{1\})  
\]
whose character is $\GP_{(4,2)}(x_1,x_2,x_3)$.

\item Among the connected normal $\sqgln$- and $\sqqn$-crystals that do have unique highest weight elements, there exist non-isomorphic objects with the same highest weight. 
For example, \[(\{1\},\{2\},\{1\}) \in \SVWords_3(3)\]
is the unique highest weight element of a full $\sqrt{\gl_3}$-subcrystal
that is not isomorphic to $\SetTab_3((2,1))$, although both have character $G_{(2,1)}(x_1,x_2,x_3)$.
Similarly, \[(\{1\}, \{2\}, \{1\}, \{1\}) \in \SVWords_3(4)\]
is the unique highest weight element of a full $\sqrt{\q_3}$-subcrystal
that is not isomorphic to $\SetDTab_3((3,1))$, although both have character $\GP_{(3,1)}(x_1,x_2,x_3)$.

\item  We expect that the normal crystals $\SetTab_n(\lambda)$ and $\SetDTab_n(\lambda)$  have unique highest weight elements of weight $\lambda$; this is known for  $\SetTab_n(\lambda)$ but only conjectural for $\SetDTab_n(\lambda)$.
One might hope that taking the isomorphism classes 
of just these set-valued tableau crystals would give good substitutes for the ``too large'' categories of
normal $\sqgln$- and $\sqqn$-crystals. However, neither of the resulting subcategories is closed under tensor products.
For example,     
\[ 
\SetTab_3((1))\otimes \SetTab_3((2))
\quand
\SetDTab_3((1))\otimes \SetDTab_3((2))
\]
each have a full subcrystal of unique highest weight $(3,1)$.
The subcrystal of the first object is 
not isomorphic to $\SetTab_3((3,1))$ as a $\sqrt{\gl_3}$-crystal,
and the subcrystal of the second object is
not isomorphic to $\SetDTab_3((3,1))$ as a $\sqrt{\q_3}$-crystal.

\end{itemize}

Despite these observations, 
normal $\sqgln$- and normal $\sqqn$-crystals
do appear to have some interesting properties, and 
there seems to be a close connection between
these crystal categories and
the rings of symmetric functions generated by $\{G_\lambda\}$ and $\{\GP_\lambda\}$.
This speculative relationship is indicated by the conjectures explained in the next section.

\subsection{Conjectures}\label{conj-sect}

We use this  final section to present several conjectures about normal $\sqgln$- and $\sqqn$-crystals 
that would be interesting to explore in future work.

First, we say that an element of $\ZZ[x_1,\dots,x_n]$
is \defn{$G$-positive}
(respectively, \defn{$\GP$-positive})
if it is an $\NN$-linear combination of 
 polynomials of the form $G_\lambda(x_1,\dots,x_n)$
(respectively, $\GP_\lambda(x_1,\dots,x_n)$). 

\begin{example}
The character of any tensor power of  $\SS_n$
is both  $G$- and $\GP$-positive since 
\[ \ch(\SS_n^{\otimes m}) = \ch(\SS_n)^m = G_{(1)}(x_1,x_2,\dots,x_n)^m
= \GP_{(1)}(x_1,x_2,\dots,x_n)^m\]
and the Pieri rules in \cite{Buch2002,BuchRavikumar}
expand the last expression as a sum of $G$- or $\GP$-functions.
\end{example}

Although we have $\ch(\SetTab_n(\lambda)) = G_\lambda(x_1,\dots,x_n)$
and, conjecturally,
$\ch(\SetDTab_n(\lambda)) = \GP_\lambda(x_1,\dots,x_n)$,
the characters of other connected normal $\sqgln$- and $\sqqn$-crystals
are not single $G$- or $\GP$-polynomials. However, the following holds in all examples we can compute:

\begin{conjecture}
The character of any finite normal $\sqgln$-crystal is  
$G$-positive.
\end{conjecture}

\begin{conjecture}
The character of any finite normal $\sqqn$-crystal is
$\GP$-positive.
\end{conjecture}

A more precise version of the first conjecture, generalizing Theorem~\ref{buch-cor}, also seems to hold:
\begin{conjecture}\label{cBb-conj}
If $\cB$ is a finite normal $\sqgln$-crystal then
$\ch(\cB) = \sum_b G_{\weight(b)}(x_1,\dots,x_n)$
where the sum is over all $b \in \cB$
with $e'_i(b) =0$ for all $i \in [n-1]$.
\end{conjecture}

We have checked these conjectures by computer for all subcrystals of 
$\SS_n^{\otimes m}\cong \SVWords_n(m)$ when 
$n=2$ and $m\leq 12$,
$n=3$ and $m\leq 7$,
$n=4$ and $m\leq 5$,
$n=5$ and $m\leq 4$, or 
$n=6$ and $m\leq 3$.
One might try to prove them
by identifying set-valued analogues of 
\defn{RSK insertion} or Haiman's \defn{shifted mixed insertion} \cite{HaimanMixed},
perhaps extending the \defn{uncrowding algorithms} in \cite{Buch2002,PPPS2022}.

As in Remark~\ref{buch-rmk}, it is not clear how to convert Conjecture~\ref{cBb-conj} 
to a formula for the $\GP$-expansion of the character of
 a finite normal $\sqqn$-crystal.
 Like Theorem~\ref{buch-cor},
the character formula in Conjecture~\ref{cBb-conj}  does not lift to an isomorphism of $\sqgln$-crystals.

 We highlight the special case of Conjecture~\ref{cBb-conj} with $\cB =  \SetDTab_n(\lambda)$.
This asserts that the $\sqgln$-highest weight elements in
 $\SetDTab_n(\lambda)$
 encode the $G$-expansion of its character,
 which via Conjecture~\ref{Ik-conj} is expected to be 
 $\GP_\lambda(x_1,x_2,\dots, x_n)$. The coefficients in the $G$-expansion of $\GP_\lambda$ are already known to be nonnegative integers by \cite[Thm.~3.27]{MarScr}.

\begin{conjecture}\label{GPG-conj}
If $\lambda$ is a strict partition then
\[\GP_\lambda(x_1,x_2,\dots,x_n) =
\sum_{
\substack{T \in \SetDTab_n(\lambda)  \\
e'_i(T) =0\ \forall i \in[n-1]}}
G_{\weight(T)}(x_1,x_2,\dots,x_n).
\]
\end{conjecture}

We have checked this conjecture  when
 $n=2$ and $|\lambda| \leq 30$,
$n=3$ and $|\lambda| \leq 12$,
$n=4$ and $|\lambda| \leq 8$,
$n=5$ and $|\lambda| \leq 6$, or
$n=6$ and $|\lambda| \leq 4$.
 
Recall the definition of $w(T)$ for $T \in \SVWords_n(m)$ from Lemma~\ref{sv-highest-lem}.
Via that result and Theorem~\ref{sqqn-thm2},
Conjecture~\ref{GPG-conj} is equivalent to the following statement: 
 
\begin{conjecture}
If $\lambda$ is a strict partition then $\GP_\lambda = \sum_\nu d_{\lambda\mu} G_\mu$ where the sum is over strict partitions $\mu$
and $d_{\lambda\mu}$ is the number of set-valued decomposition tableaux $T$
of shape $\lambda$
and weight $\weight(T) =\mu$ such that $w(\svrevrow(T))$ is a reverse lattice word.
\end{conjecture}

The conjectures above are plausible in view of the following observations.
We have seen that the character of any finite $\sqgln$-crystal is symmetric.
It is well-known that any symmetric polynomial in $\ZZ[x_1,\dots,x_n]$ is a $\ZZ$-linear combination of 
symmetric Grothendieck polynomials. In addition:

\begin{proposition}\label{sqrt-char-prop}
The character of any finite normal $\sqqn$-crystal is symmetric with the $K$-theoretic $Q$-cancelation property,
 so is a $\ZZ$-linear combination of $\GP$-polynomials  $\GP_\lambda(x_1,\dots,x_n)$.
\end{proposition}

\begin{proof}
We prove this when $\cB$ is a connected normal $\sqgln$-crystal.
We may assume that $\cB \subseteq \SVWords_n(m)$ for some $m$. We already know that $\ch(\cB)$
is symmetric by Proposition~\ref{sq-prop} so we just need to verify that it has the  $K$-theoretic $Q$-cancelation property.

Let $\cB^\ast$ be the set of \defn{multiset-valued words}
$S = (S_1,S_2,\dots,S_m)$ with the following properties: (a) in each multiset $S_i$
only the number $1$ may be repeated, and (b) when the repeated $1$'s are all replaced by a single copy, one obtains a set-valued word $\overline S \in\cB$. 
Then 
 \[ 
 \ch(\cB)(\tfrac{x_1}{1-x_1}, -x_2,x_3,\dots,x_n) =  \sum_{S \in \cB^\ast} (-1)^{c_2(S)} x_1^{c_1(S)} x_2^{c_2(S)} x_3^{c_3(S)} \cdots x_n^{c_n(S)}
 \]
 where $c_j(S)$ denotes the multiplicity of $j$ in the multiset union $S_1\cup S_2 \cup \dots \cup S_m$.
 
  As in the proof of Proposition~\ref{kca-prop}, to show that $\ch(\cB)$ has the $K$-theoretic $Q$-cancelation property, we must demonstrate that the right side of the previous equation is independent of $x_1$ when we set $x_1=x_2$.
  For this, it is enough to produce an involution of $\phi:\cB^\ast\to\cB^\ast$ 
  such that $\phi(S) = S$ when $c_1(S) =c_2(S)=0$
  and such that for all other $S \in \cB^\ast$ we have 
\begin{itemize}
\item $c_2(\phi(S)) = c_2(S) \pm 1$,
\item $c_1(\phi(S)) + c_2(\phi(S)) = c_1(S) + c_2(S)$, and 
 \item $ c_j(\phi(S)) = c_j(S)$ for $2 <j<n$.
 \end{itemize}
  Here is such an  involution. Suppose $i \in [m]$ is minimal with $1 \in S_i$ or $2 \in S_i$.
  If there is no such index then set $\phi(S) := S$. Otherwise, 
 form $\phi(S)$ from $S$ by modifying $S_i$ as follows.
 If there is already a (unique) copy of $2 \in S_i$, then remove this $2$ and add an extra copy of $1$.
 If $2 \notin S_i$, then instead replace one copy of $1 \in S_i$ by $2$.
 
Clearly $\phi(S)$ has the desired properties and $\phi(\phi(S))=S$.
It only remains to justify why $\phi(S) \in \cB^\ast$. Recall the definition of $\overline{S} \in \cB$
from the first paragraph of this proof. It is enough to check that $\overline{\phi(S)} \in \cB$. 
If $2 \in S_i$ and $1 \in S_i$, then $\overline{\phi(S)} = e'_{\bar 1}(\overline{S})$, which certainly belongs to $\cB$.
Likewise, if $2 \in S_i$ and $1 \notin S_i$, then  $\overline{\phi(S)} =e'_{\bar 1} (e'_{\bar 1}(\overline{S})) \in \cB$.
Finally, if $2 \notin S_i$ then 
 $\overline{\phi(S)}  = f'_{\bar 1}(\overline{S})$ when $c_1(S) > 1$ and
  $\overline{\phi(S)}  = f'_{\bar 1}( f'_{\bar 1}(\overline{S}))$ when $c_1(S) = 1$,
  and in both cases $\overline{\phi(S)} \in\cB$ as needed.
\end{proof}

Finally, we recall
from  \cite[Ex.~5.2]{BumpSchilling}
 that
if $\cB$ is a $\gl_n$-crystal then its \defn{Lusztig dual} $\cB^\vee$
is the $\gl_n$-crystal on the same underlying set
with weight map $\weight^\vee(b) := (\weight(b)_n,\dots,\weight(b)_2,\weight(b)_1)$
and crystal operators $e_i^\vee := f_{n-i}$ and $f_i^\vee := e_{n-i}$.
We define the \defn{Lusztig dual} of a $\sqgln$-crystal in the same way,
setting ${e'_i}^\vee := f'_{n-i}$ and ${f'_i}^\vee := e'_{n-i}$.

It easy to check that the operators on $\cB^\vee$ satisfy the relevant $\gl_n$- or $\sqgln$-crystal axioms, and
that the map $b\otimes c \mapsto c\otimes b$
is a crystal isomorphism  $(\cB \otimes \cC)^\vee \cong \cC^\vee \otimes \cB^\vee$.
Moreover, one always has $\BB_n \cong \BB_n^{\vee}$ as $\gl_n$-crystals and $\SS_n \cong \SS_n^\vee$
as $\sqgln$-crystals.
It follows as a consequence that:

\begin{proposition}
If $\cB$ is a connected normal $\gl_n$- or $\sqgln$-crystal, then so is $\cB^\vee$.
\end{proposition}

Any connected normal $\gl_n$-crystal is isomorphic to its Lusztig dual, which has the same unique highest weight.
This property does not hold for all connected normal $\sqgln$-crystals,
but does apply in the following cases.

A partition $\lambda$ is a \defn{rectangle} if its nonzero parts all have the same size, and a \defn{fat hook} if it has exactly two distinct part sizes,
so that
$|\{\lambda_1,\lambda_2,\dots\}\setminus\{0\}| =2$.

\begin{proposition}
Suppose $\lambda$ is a partition with $\ell(\lambda) \leq n$.
Then
 $\SetTab_n(\lambda) \cong \SetTab_n(\lambda)^\vee$ as $\sqgln$-crystals
 if 
  $\lambda$ is a rectangle, or if $\lambda$ is a fat hook with $\ell(\lambda) = n$.
 \end{proposition}

\begin{proof}
First assume that $\lambda$ is a rectangle. Given $T \in \SetTab_n(\lambda)$
form $ T^\vee$  by rotating the tableau $180^\circ$ and then applying the substitution $S \mapsto n+1-S$ to each set-valued entry.
The resulting map $T \mapsto T^\vee$ is an involution of $\SetTab_n(\lambda)$.
(When $T$ is not set-valued, this operation 
is the classical \defn{Sch\"utzenberger involution} \cite[\S2.4]{Lenart2007}.)
The set-valued column reading word of $T^\vee$ is formed by reversing $\svcol(T)$
and then applying the substitution $S \mapsto n+1-S$ to each entry.
Once we make this observation, checking that $T \mapsto T^\vee$ is an isomorphism 
$\SetTab_n(\lambda) \cong \SetTab_n(\lambda)^\vee$ is a straightforward exercise.

Now suppose that $\lambda$ is a fat hook with $\ell(\lambda) = n$.
Let $m\geq 1$ be the number of columns in the diagram of $\lambda$
with $n$ boxes. Then each $T \in \SetTab_n(\lambda)$, being semistandard,
must have $T_{ij} = \{i\}$ for all $(i,j)\in [n]\times [m]$,
and removing these boxes yields an arbitrary element $U \in \SetTab_n(\mu)$
where $\mu$ is the rectangular partition $(\lambda_1-m,\lambda_2-m,\dots,\lambda_n - m)$.
Define $T^\vee$ from $T$ by leaving the first $m$ columns unchanged
and replacing $U$ by $U^\vee$ as given in the previous paragraph.
It is again a simple exericise to check that   $T \mapsto T^\vee$ is a 
crystal isomorphism  $\SetTab_n(\lambda) \cong \SetTab_n(\lambda)^\vee$.
\end{proof}

Our computations suggest that the converse to the previous result holds.

\begin{conjecture}
Suppose $\lambda$ is a partition with $\ell(\lambda) \leq n$.
Then
$\SetTab_n(\lambda) \cong \SetTab_n(\lambda)^\vee$ 
as $\sqgln$-crystals
only if  $\lambda$ is a rectangle,
or   $\lambda$ is a fat hook
with $\ell(\lambda)=n$.
\end{conjecture}

We require $\lambda$ to have $\ell(\lambda)\leq n$ in these statements
since otherwise $\SetTab_n(\lambda)$ is empty.
We have checked by computer that this conjecture holds 
at least 
for 
$n\leq 5$ and $|\lambda|\leq 12$.

\printbibliography

\end{document}